\documentclass[12pt]{article}
\usepackage[utf8]{inputenc}
\usepackage{soul}
\usepackage{hyperref}
\usepackage{authblk}
\usepackage[T1]{fontenc}
\usepackage{amsmath, amsthm, amsfonts}
\usepackage[initials]{amsrefs}
\usepackage{mathtools}
\usepackage[english]{babel}
\usepackage{amssymb}
\usepackage{csquotes}
\usepackage{enumerate}
\usepackage{mathrsfs}
\usepackage{bbm}
\usepackage[
top    = 1.75cm,
bottom = 2.25cm,
left   = 2.25cm,
right  = 2.25cm]{geometry}
\usepackage{fancyhdr}
\usepackage{longtable}
\usepackage{multicol}
\usepackage{holtpolt}
\usepackage{graphicx}
\usepackage{caption}
\usepackage[font=small]{caption}
\usepackage{lipsum}
\usepackage{mwe}
\usepackage{subcaption}
\usepackage{xcolor}
\usepackage{tikz-cd}
\usepackage{booktabs}
\usepackage{lmodern}
\usepackage{microtype}
\usepackage{enumerate}
\usepackage{float}
\hypersetup{
	colorlinks,
	linkcolor={blue!80!black},
	citecolor={blue!80!black},
	urlcolor={blue!80!black}
}

\usepackage[colorinlistoftodos]{todonotes} %Use to show the notes

\newtheoremstyle{mytheoremstyle1} % name
    {\topsep}                    % Space above
    {\topsep}                    % Space below
    {\itshape}                   % Body font
    {}                           % Indent amount
    {\bf}                   % Theorem head font
    {}                          % Punctuation after theorem head
    {.7em}                       % Space after theorem head
    {}  % Theorem head spec (can be left empty, meaning ‘normal’)

\newtheoremstyle{mytheoremstyle2} % name
    {\topsep}                    % Space above
    {\topsep}                    % Space below
    {}                   		% Body font
    {}                           % Indent amount
    {}                   % Theorem head font
    {}                          % Punctuation after theorem head
    {.7em}                       % Space after theorem head
    {}  % Theorem head spec (can be left empty, meaning ‘normal’)

\theoremstyle{mytheoremstyle1}
\newtheorem{thm}{Theorem}[section]

\newtheorem{prop}[thm]{Proposition}

\theoremstyle{mytheoremstyle2}

\theoremstyle{remark}
\newtheorem{rem}[thm]{Remark}

\makeatletter \@addtoreset{equation}{section} \makeatother

\def\R{\mathbb{R}}

\def\ZZ{\mathbb{Z}}

\def\NN{\mathbb{N}}

\newcommand\dif{\mathop{}\!\mathrm{d}}

\def\O{\mathcal{O}}

\DeclareMathOperator{\RE}{Re}

\DeclareMathOperator{\D}{D}
\DeclareMathOperator{\ii}{i}

\DeclareMathOperator{\per}{per}
\DeclareMathOperator{\loc}{loc}
\DeclareMathOperator{\TT}{\text{T}}

\def\eps{\varepsilon}

\def\i{\text{i}}

\begin{document}
	\title{Transverse dynamics of two-dimensional traveling\\ periodic gravity--capillary water waves
		\footnotesize{\author{Mariana Haragus, Tien Truong \& Erik Wahlén}
	}}
	\date{}
	\newcommand{\Addresses}{{% additional braces for segregating \footnotesize
			\bigskip
			\footnotesize
			M.~Haragus, \textsc{FEMTO-ST institute, Univ. Bourgogne Franche-Comt\'e, France}\par\nopagebreak
			\textit{E-mail address:} \texttt{mariana.haragus@femto-st.fr}
			
			\medskip
			
			T.~Truong, \textsc{Centre for Mathematical Sciences, Lund University, Sweden}
			\par\nopagebreak
			\textit{E-mail address:} \texttt{tien.truong@math.lu.se}
			
			\medskip
			
			E.~Wahl\'en, \textsc{Centre for Mathematical Sciences, Lund University, Sweden}\par\nopagebreak
			\textit{E-mail address:} \texttt{erik.wahlen@math.lu.se}

	}}
	\maketitle
	
	\begin{abstract}
		We study the transverse dynamics of two-dimensional traveling periodic waves for the gravity--capillary water-wave problem. The governing equations are the Euler equations for the irrotational flow of an inviscid fluid layer with free surface under the forces of gravity and surface tension. We focus on two open sets of dimensionless parameters $(\alpha,\beta)$, where $\alpha$ and $\beta$ are the inverse square of the Froude number and the Weber number, respectively. For each arbitrary but fixed pair  $(\alpha,\beta)$ in one of these sets, two-dimensional traveling periodic waves bifurcate from the trivial constant flow. In one open set we find a one-parameter family of periodic waves, whereas in the other open set we find two geometrically distinct one-parameter families of periodic waves. Starting from a transverse spatial dynamics formulation of the governing equations, we investigate the transverse linear instability of these periodic waves and the induced dimension-breaking bifurcation. The two results share a common analysis of the purely imaginary spectrum of the linearization at a periodic wave. We apply a simple general criterion for the transverse linear instability problem and a Lyapunov center theorem for the dimension-breaking bifurcation. For parameters  $(\alpha,\beta)$ in the open set where there is only one family of periodic waves, we prove that these waves are linearly transversely unstable. For the other open set, we show that the waves with larger wavenumber are transversely linearly unstable. We also identify an open subset of parameters for which both families of periodic waves are tranversely linearly unstable. For each of these transversely linearly unstable periodic waves, a dimension-breaking bifurcation occurs in which three-dimensional doubly periodic waves bifurcate from the two-dimensional periodic wave.
	\end{abstract}
	
	
	\section{Introduction}
	\label{sect-hydro}
	
	We consider a three-dimensional inviscid fluid with constant density $\rho$ occupying a region
	\[D_{\eta}=\{(X, Y, z) \in \R^3 \, : \, 0 < Y < h+\eta(X, z, t)\},\]
	where $(X,Y,z)$ are Cartesian coordinates, $h$ is the mean fluid depth, and $\eta > -h$ is the unknown free surface of the fluid depending on the horizontal spatial variables $X,z$ and the time variable~$t$. The fluid is under the influence of the gravitational force with acceleration constant $g$ and surface tension with coefficient $T$. We assume that the flow is irrotational and denote by $\phi$ an Eulerian velocity potential. Choosing a coordinate frame moving from left to right along the $X$-axis with constant velocity $c>0$, the fluid motion is described by Laplace's equation
	\begin{equation}
	\label{eq-hydro-nondim}
	\phi_{XX}+\phi_{YY}+\phi_{zz} = 0 \quad \text{for} \quad 0 < Y < 1+\eta,
	\end{equation}
	with boundary conditions
	\begin{equation}
	\label{eq-hydro-bc-nondim}
	\begin{aligned}
	&\phi_Y = 0& \quad &\text{on $Y=0$},&\\
	&\phi_Y = \eta_t - \eta_X + \eta_X \phi_X + \eta_z \phi_z& \quad &\text{on $Y=1+\eta$},&\\
	&\phi_t -  \phi_X + \frac{1}{2}\left(\phi_X^2+\phi_Y^2+\phi_z^2\right) + \alpha\eta - \beta \mathcal K = 0& &\text{on $Y=1+\eta$}.&
	\end{aligned}
	\end{equation}
	Here, we have used dimensionless variables by taking the characteristic length scale $h$ and characteristic time scale $h/c$. The dimensionless parameters
	\[\quad \alpha = \frac{gh}{c^2}  \quad \text{and} \quad  \beta = \frac{T}{\rho hc^2}\]
	are the inverse square of the Froude number and the Weber number, respectively, and the quantity $\mathcal K$ is twice the mean curvature of the free surface $\eta$, given by
	\[\mathcal K = {\left[\frac{\eta_X}{\sqrt{1+\eta_X^2+\eta_z^2}}\right]}_X + \left[\frac{\eta_z}{\sqrt{1+\eta_X^2+\eta_z^2}}\right]_z.\]
	The set of equations \eqref{eq-hydro-nondim}--\eqref{eq-hydro-bc-nondim} are the Euler equations for gravity--capillary waves on water of finite depth. The case $\beta=0$, that we do not consider in this work, corresponds to gravity water waves. 
	
	We are interested in the transverse dynamics of two-dimensional traveling periodic waves. In the above formulation these are steady solutions which are periodic in $X$ and do not depend on the second horizontal coordinate $z$ and on the time $t$. Their existence is well-recorded in the literature; e.g., see \cite{DI03,Groves04, Groves07, HI} and the references therein. Many of these results are obtained using methods from bifurcation theory. Bifurcations of two-dimensional periodic waves are determined by the positive roots of the linear dispersion relation
	\begin{equation}\label{eq-lindisrel}
	\mathcal D(k) \coloneqq (\alpha + \beta k^2) \sinh|k| - |k|\cosh(k)=0,
	\end{equation}
	obtained by looking for nontrivial solutions of the form $(\eta(X),\phi(X,Y))=(\eta_{k},\phi_k(Y)) \exp(\i kX)$ to the steady system  \eqref{eq-hydro-nondim}--\eqref{eq-hydro-bc-nondim} linearized at $(\eta_0(X), \phi_0(X,Y))=(0,0)$.
	Associated to any positive root $k$ of the linear dispersion relation, one finds a one-parameter family of periodic waves $\{(\eta_\varepsilon(X),\phi_\varepsilon(X,Y))\}_{\varepsilon\in(-\varepsilon_0,\varepsilon_0)}$ with wavenumbers close to $k$, for sufficiently small $\varepsilon_0>0$. These periodic waves bifurcate from the trivial solution $(\eta_0,\phi_0)=(0,0)$.
	
	Depending on the values of the two parameters $\alpha$ and $\beta$, the linear dispersion relation \eqref{eq-lindisrel} possesses positive roots in the following cases:
	\begin{enumerate}
		\item one positive simple root $k_*>0$  if  $\alpha \in (0, 1)$ and $\beta > 0$; we refer to this set of parameters as Region~I.
		\item two positive simple roots $0<k_{*,1}<k_{*,2}$ if $\alpha > 1$ and $0<\beta<\beta(\alpha)$, where $(\alpha,\beta(\alpha))$ belongs to the curve $\Gamma$ with parametric equations
		\begin{equation}\label{eq-Gamma}
		\alpha = \frac{s^2}{2 \sinh^2(s)} + \frac{s}{2 \, \tanh(s)},\quad
		\beta = -\frac{1}{2\sinh^2(s)} + \frac{1}{2s \, \tanh(s)} , \quad s \in (0, \infty) ;
		\end{equation}
		we refer to this set of parameters as Region~II;
		\item  one positive simple root $k_*>0$  if  $\alpha =1$ and $\beta<1/3$;
		\item one positive double root $k_*$ if $(\alpha,\beta)$ belongs to the curve $\Gamma$ given in \eqref{eq-Gamma}.
	\end{enumerate}
	The linear dispersion relation being even in $k$, together with any positive root $k$ we also find the negative root $-k$. We illustrate these properties in the left panel of Figure~\ref{fig-lindis}.  
	\begin{figure}[H]
		\begin{center}
			\includegraphics[scale=0.8]{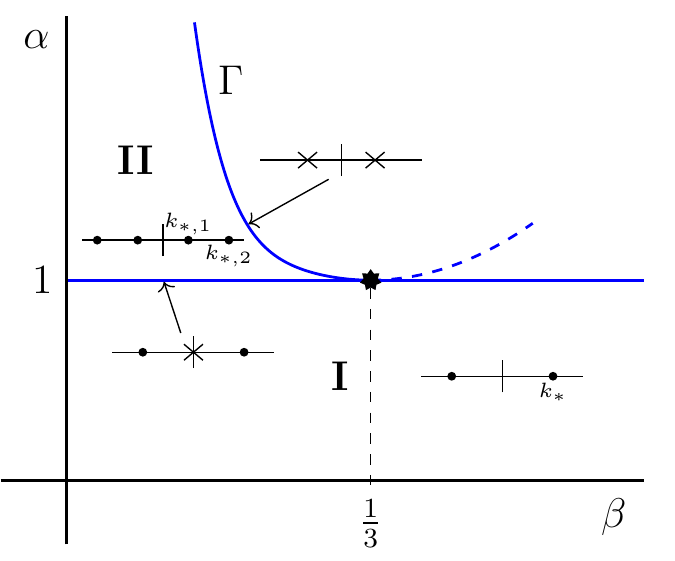}
			\includegraphics[scale=0.8]{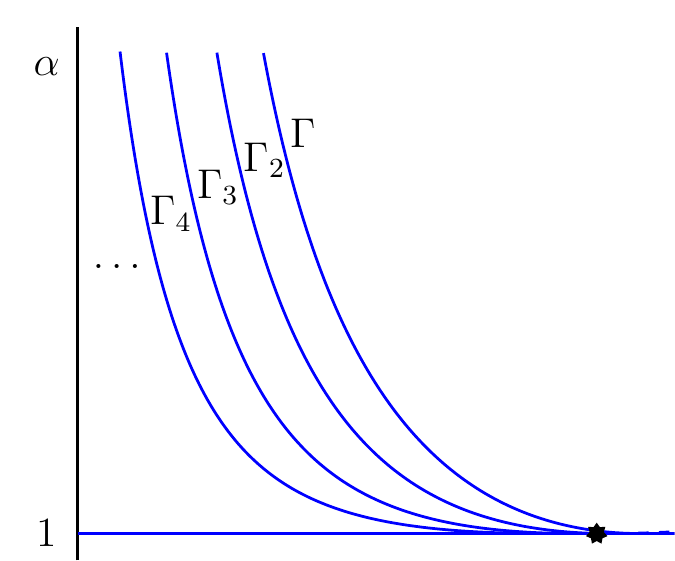}
			\captionsetup{width=.8\linewidth}
			\caption{Left: In the $(\beta,\alpha)$-plane, sketch of the nonzero roots of the linear dispersion relation~\eqref{eq-lindisrel}. We use dots to indicate simple roots and crosses to indicate double roots.
				Right: In Region~II, plot of the curves $\Gamma_m$ for $m=2,3,4$ which are excluded from our analysis.}
			\label{fig-lindis}
		\end{center}
	\end{figure}
	Here, we focus on the periodic waves which bifurcate in the two open parameter regions I and II.
	For simplicity, in Region~II we assume that $k_{*,1}$ and $k_{*,2}$ satisfy the non-resonance condition $k_{*,2}/k_{*,1} \notin \NN$. This assumption means that $(\alpha,\beta)$ does not belong to any of the curves $\Gamma_m$ for $m \in \NN$, $m \ge 2$, with parametric equations
	\[
	\begin{array}{l}
	\alpha = \displaystyle{-\frac{m^2 s}{(1-m^2)\tanh(s)}+\frac{ms}{(1-m^2)\tanh( ms)}}\\[2ex]
	\beta = \displaystyle{\frac{1}{(1-m^2)s \tanh(s)}- \frac{m}{(1-m^2)s \tanh(ms)}}
	\end{array}, \quad s \in (0, \infty);\]
	see the right panel in Figure~\ref{fig-lindis}. Then, for any fixed $(\alpha,\beta)$ in Region~I, there is a one-parameter family of two-dimensional periodic waves $\{(\eta_\varepsilon(X),\phi_\varepsilon(X,Y))\}_{\varepsilon\in(-\varepsilon_0,\varepsilon_0)}$ with wavenumbers close to $k_*$, whereas for $(\alpha,\beta)$ in Region~II, there are two geometrically distinct families of periodic waves 
	\[\{(\eta_{\eps,1}(X), \phi_{\eps,1}(X, Y))\}_{\eps \in (-\eps_0, \eps_0)}\, \, \text{and} \, \, \{(\eta_{\eps,2}(X), \phi_{\eps, 2}(X, Y))\}_{\eps \in (-\eps_0, \eps_0)}\] 
	with wavenumbers close to $k_{*,1}$ and $k_{*,2}$, respectively.
	
	The purpose of our transverse dynamics analysis is twofold: to identify the periodic waves in regions I and II which are transversely linearly unstable and to discuss the induced dimension-breaking bifurcations. Roughly speaking, a two-dimensional wave is transversely linearly unstable if the Euler equations  \eqref{eq-hydro-nondim}--\eqref{eq-hydro-bc-nondim} linearized at the wave possess solutions which are bounded in the  horizontal coordinates $(X,z)$ and exponentially growing in time $t$. The dimension-breaking bifurcation is the bifurcation of three-dimensional solutions emerging from the two-dimensional transversely unstable wave. Typically, these three-dimensional solutions are periodic in the transverse horizontal coordinate $z$; see Figure~\ref{fig-2D-sol} for an illustration in the case of a two-dimensional periodic wave. Though of different type, these two questions share a common spectral analysis of the linear operator at the two-dimensional wave. This is the key, and most challenging, part of our analysis.
	\begin{figure}[H]
		\begin{center}
			\includegraphics[scale=0.6]{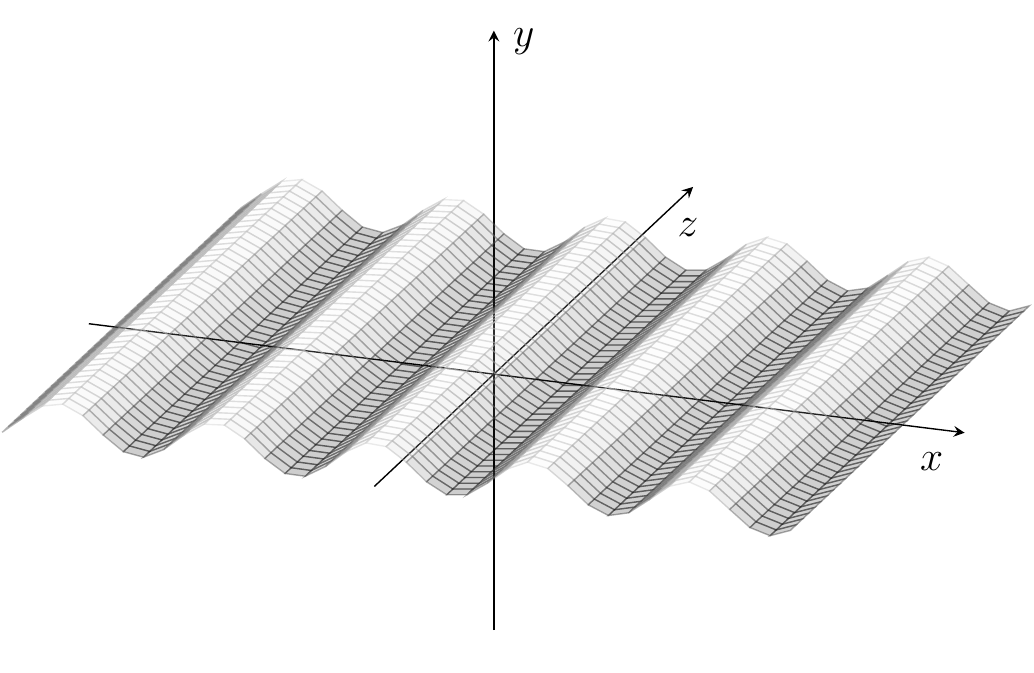}
			\includegraphics[scale=0.6]{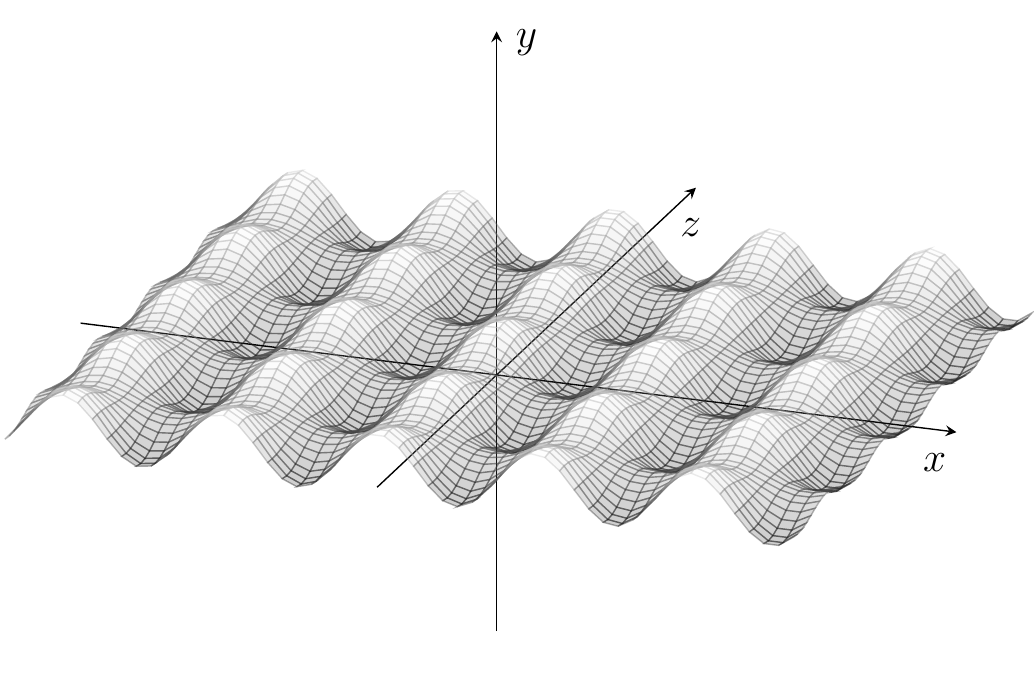} 
			\captionsetup{width=.8\linewidth}
			\caption{Illustration of a dimension-breaking bifurcation. Left: plot of a two-dimensional periodic wave. Right: plot of a bifurcating three-dimensional doubly periodic wave.}
			\label{fig-2D-sol}
		\end{center}        
	\end{figure}
	
	The transverse stability of periodic waves was mostly studied for simpler model equations obtained from the Euler equations \eqref{eq-hydro-nondim}--\eqref{eq-hydro-bc-nondim} in different parameter regimes: the Kadomtsev--Petviashvili-I equation for the regime of large surface tension ($\alpha\sim1,\ \beta>1/3$) was considered in \cite{Haragus2,JZ,HSS}, the Davey--Stewartson system for the regime of weak surface tension ($(\alpha,\beta)$ close to the curve $\Gamma$) in \cite{Godey}, and a fifth order KP equation for the regime of critical surface tension ($\alpha\sim1,\ \beta\sim1/3$) in \cite{HW}; see also the recent review paper \cite{H19}. All these results predict that gravity--capillary periodic waves are linearly transversely unstable. We point out that pure gravity periodic, or solitary, water waves ($\beta=0$) are expected to be linearly transversely stable \cite{APS,HLP}.
	
	For the Euler equations, previous works on transverse instability mostly treat the case of solitary waves; see \cite{GHS,PS,RT2011} for the large surface regime and the more recent work \cite{GSW} for the weak surface tension regime close to the curve $\Gamma$. In both regimes, the dimension-breaking bifurcation has been studied in \cite{GHS02} (large surface tension) and \cite{GSW} (weak surface tension).
	For periodic waves, the transverse instability predicted in the regime of large surface tension ($\alpha\sim1,\ \beta>1/3$) has been confirmed in \cite{Haragus}. In addition, the dimension-breaking bifurcation was studied showing the bifurcation of a one-parameter family of three-dimensional doubly periodic waves, as illustrated in Figure~\ref{fig-2D-sol}. In the present work, we treat these two questions for the periodic waves bifurcating in the open parameter regions I and II. 
	
	For completeness, we mention that there are other stability/instability results for these periodic waves. When the perturbations are constant in $z$, the references~\cite{Kawahara, DR, EM} through formal expansions have provided a characterization for the Benjamin--Feir instability\footnote{This is linear instability with respect to sideband perturbations, which has a different period than that of the main periodic wave. It was first discovered for gravity waves in deep fluids by Benjamin \& Feir~\cite{BF}, Benjamin~\cite{Benjamin} and independently by Whitham~\cite{Whitham2}.} (see e.g.~\cite{BridgesMielke, ChenSu, HurYang, NguyenStrauss} for rigorous proofs), the work~\cite{DT} demonstrates numerically that periodic waves are sometimes spectrally unstable even when the Benjamin--Feir instability is not present, references~\cite{McGoldrick1, McGoldrick2, McGoldrick3, Jones1, Jones2, TDW} indicate through both numerical and theoretical investigations that harmonic resonances feature even more intriguing instability phenomena, e.g. nested instabilities or multiple high-frequency instability bubbles. The works~\cite{HJ, HP17, HP} approach these questions through their own proposals of fully dispersive model equations. In particular,~\cite{HP17, HP} find qualitatively the same instability characterization for periodic waves in their models as~\cite{Kawahara, DR}. Instability under three-dimensional perturbations has been considered numerically, experimentally and using various model equations, such as the Davey--Stewartson equation~\cite{DS, DR, Hayes, HH, ZM}. In particular, the instability criterion that we arrive at here can be formally obtained by taking $l=0$ in equation (3.9) in~\cite{DS}, and using the formulas for the coefficients for gravity--capillary waves in~\cite{AblowitzSegur79, DR}; see Appendix~\ref{app-formal-derivation}. Note that in contrast to the previous studies, we restrict our attention to perturbations which have the same wavelength as the periodic wave in the $X$-direction.

	\color{black}
	
	Our approach to transverse dynamics follows the ideas developed for solitary waves in \cite{GHS,GHS02,GSW}. The starting point of the analysis is a spatial dynamics formulation of the three-dimensional, time-dependent equations \eqref{eq-hydro-nondim}--\eqref{eq-hydro-bc-nondim} in which the horizontal coordinate $z$, transverse to the direction of propagation, plays the role of time.
	
	For the transverse linear instability problem, we consider the linearization of this dyna- mical system at a two-dimensional periodic wave and apply a simple general instability criterion \cite{Godey} adapted to the Euler equations in \cite{GSW}. In Region~I, we show that the periodic waves $\{(\eta_\varepsilon(X),\phi_\varepsilon(X,Y))\}_{\varepsilon\in(-\varepsilon_0,\varepsilon_0)}$ are transversely linearly unstable, provided $\varepsilon_0$ is sufficiently small. In Region~II, we obtain transverse linear instability for the periodic waves $\{(\eta_{\eps,2},\phi_{\eps, 2})\}_{\eps \in (-\eps_0, \eps_0)}$ with wavenumbers close to the largest root $k_{*,2}$ of the linear dispersion relation. For the second family of periodic waves,  $\{(\eta_{\eps,1},$ $\phi_{\eps, 1})\}_{\eps \in (-\eps_0, \eps_0)}$ with wavenumbers close to $k_{*,1}$, we can only conclude on transverse instability for parameter values $(\alpha, \beta)$ situated in the open region between the curves $\Gamma$ and $\Gamma_2$; see the right panel in Figure~\ref{fig-lindis}.
	The dimension-breaking bifurcation is studied for the transversely linearly unstable periodic waves. Here, we use the time-independent, but nonlinear, version of the dynamical system above. Applying a Lyapunov center theorem, we prove that from each unstable periodic wave bifurcates a family of doubly periodic waves. 
	
	The common part of the proofs of these two results is the analysis of the purely imaginary spectrum of the linearized operator at the two-dimensional periodic wave. This analysis is the major part of our work. Our main result shows that this linear operator possesses precisely one pair of simple nonzero purely imaginary eigenvalues. Though it relies upon standard perturbation arguments for linear operators, the proof is rather long because of the complicated formulas for the linear operator. This spectral result is the key property allowing to apply both the transverse instability criterion and the Lyapunov center theorem.
	
	In the following theorem, we summarize the results obtained for Region~I. 
	\begin{thm}[Region~I]
		\label{thm-main-II}
		Fix $(\alpha,\beta)$ in Region~I and let $k_* > 0$ be the unique positive root of the linear dispersion relation \eqref{eq-lindisrel}. 
		\begin{enumerate}[(i)]
			\item \label{thm-main-II-exist} (Existence) There exist $\eps_0 > 0$ and a one-parameter family of two-dimensional steady solutions $\{(\eta_\eps(X), \phi_\eps(X, Y))\}_{\eps \in (-\eps_0, \eps_0)}$ to equations~\eqref{eq-hydro-nondim}--\eqref{eq-hydro-bc-nondim}, such that $(\eta_0,\phi_0) =(0, 0)$ and $(\eta_\eps, \phi_\eps)$ are periodic in $X$ with wavenumber $k_\eps = k_* + \O(\eps^2)$.
			\item \label{thm-main-II-unstable} (Transverse instability) There exists $\eps_1>0$ such that for each $\eps \in (-\eps_1, \eps_1)$ the periodic solution $(\eta_\eps(X), \phi_\eps(X,Y))$ is transversely linearly unstable.
			\item \label{thm-main-II-dimbreak} (Dimension-breaking bifurcation) There exists $\eps_2>0$, such that for each $\eps \in (-\eps_2, \eps_2)$ there exist $\delta_\eps > 0$, $\ell_\eps^*>0$, and a one-parameter family of three-dimensional doubly periodic waves $\{(\eta_\eps^{\delta}(X, z),\phi_\eps^{\delta}(X,Y,z))\}_{\delta \in (-\delta_\eps, \delta_\eps)}$, with wavenumber $k_\eps$ in $X$ and wavenumber $\ell_\delta = \ell_\eps^* + \O(\delta^2)$ in $z$, bifurcating from the periodic solution $(\eta_\eps(X),$ $\phi_\eps(X, Y))$. 
		\end{enumerate}
	\end{thm}
	
	We point out that $\pm {\ii}\ell_\eps^*$ where $\ell_\eps^*>0$ is given in Theorem~\ref{thm-main-II}(\ref{thm-main-II-dimbreak}) are the two nonzero purely imaginary eigenvalues of the linearization at the periodic wave. The results found for Region~II are summarized in Theorem~\ref{thm-main-III} from Section~\ref{sect-reg-III}.
	
	In our presentation we focus on Region~I, the arguments being, up to some computations, the same for Region~II. In Section~\ref{s prelim} we recall the spatial dynamics formulation of the three-dimensional time-dependent Euler equations \eqref{eq-hydro-nondim}--\eqref{eq-hydro-bc-nondim} from \cite{GSW} and the existence result for two-dimensional periodic waves given in Theorem~\ref{thm-main-II}(\ref{thm-main-II-exist}). We also give some explicit expansions of these solutions which are computed in Appendix~\ref{app-2D-calc}. In Section~\ref{s linear} we prove the results for the linear operator. Some of the long computations needed here are given in Appendices~\ref{app-m212} and~\ref{app-m212-sign}. In Section~\ref{s transverse} we present the transverse dynamics results and in Section~\ref{sect-reg-III} we discuss the results for Region~III. Finally, in Appendix~\ref{app-Lyapunov} we recall an infinite-dimensional version of the Lyapunov center theorem, and in Appendix~\ref{app-formal-derivation} we show how the instability criterion can be derived formally using the Davey-Stewartson approximation.

	\section{Preliminaries}
	\label{s prelim}
	
	In this section, we recall the spatial dynamics formulation from \cite{GSW} and
	the result on existence of two-dimensional steady periodic solutions.

	\subsection{Spatial dynamics formulation}
	\label{sect-spatial-formulation}
	
	Following \cite{GSW}, we make the change of variables
	\[
	Y=y(1+\eta(X, z, t)),\quad \phi(X,Y, z, t) = \Phi(X, y, z, t),
	\]
	in~\eqref{eq-hydro-nondim}--\eqref{eq-hydro-bc-nondim} to flatten the free surface. Since we consider periodic solutions, in addition, we set $X=kx$ with $k$ the wavenumber in $x$. We introduce two new variables,
	\[\begin{split}\omega &= -\int_0^1 \left( \Phi_z - \frac{y \eta_z \Phi_y}{1+\eta}\right) y \Phi_y \dif y + \frac{\beta \eta_z}{(1+k^2\eta_x^2+\eta_z^2)^{1/2}}, \\
	\xi &= (1+\eta) \left( \Phi_z - \frac{y \eta_z \Phi_y}{1+\eta}\right).
	\end{split}\]
	and set $U = (\eta, \omega, \Phi, \xi)^{\text{T}}$. Then, the equations~\eqref{eq-hydro-nondim}--\eqref{eq-hydro-bc-nondim} can be written as a dynamical system of the form
	\begin{equation}\label{eq-hydro-evo}
	\frac{\dif U}{\dif z} = D U_t + F(U),
	\end{equation}
	with boundary conditions
	\begin{equation} \label{eq-hydro-bc-evo} 
	\Phi_y = y \eta_t + B(U) \quad \text{on} \quad y=0,1.
	\end{equation}
	Here, $D$ is the linear operator defined by
	\[ D U=(0, \Phi|_{y=1}, 0, 0)^{\text{T}},\]
	$F$ is the nonlinear mapping $F(U)=(F_1(U), F_2(U), F_3(U), F_4(U))^{\text{T}}$ given by
	\[
	\begin{split}
	F_1(U) &= W \left(\frac{1+k^2\eta_x^2}{\beta^2-W^2}\right)^{1/2},\\
	F_2(U) &= \frac{F_1(U)}{(1+\eta)^2} \int_0^1 y \Phi_y \xi \dif y - k\left[k\eta_x \frac{W}{F_1(U)}\right]_x + \alpha \eta - k\Phi_x|_{y=1} \\
	&\quad + \int_0^1 \Bigg\{ \frac{\xi^2 - \Phi^2_y}{2(1+\eta)^2} + \frac{k^2}{2} \left( \Phi_x - \frac{y \eta_x \Phi_y}{1+\eta}\right)^2 + k^2\left[ \left(\Phi_x - \frac{y \eta_x \Phi_y}{1+\eta}\right) y \Phi_y \right]_x \\
	& \quad \quad + k^2 \left(\Phi_x - \frac{y \eta_x \Phi_y}{1+\eta}\right) \frac{y \eta_x \Phi_y}{1+\eta} \Bigg\} \dif y, \\
	F_3(U) &= \frac{\xi}{1+\eta}+ \frac{y \Phi_y}{1+\eta} F_1(U), \end{split}\]
	and
	\[\begin{split}
	F_4(U) &= -\frac{\Phi_{yy}}{1 + \eta} - k^2\left[(1+\eta)\left(\Phi_x - \frac{y \eta_x \Phi_y}{1+\eta}\right)\right]_x \hspace{4.8cm}\\
	&\quad +  k^2\left[\left(\Phi_x - \frac{y \eta_x \Phi_y}{1+\eta}\right) y \eta_x \right]_y + \frac{(y\xi)_y}{1+\eta} F_1(U),
	\end{split}
	\]
	where
	\[
	W = \omega + \frac{1}{1+\eta} \int_0^1 y \Phi_y \xi \dif y,
	\]
	and $B$ is the nonlinear mapping defined by
	\[
	B(U) = - ky\eta_x +  k^2y\eta_x \Phi_x + \frac{\eta\Phi_y}{1+\eta}- \frac{k^2y^2\eta_x^2 \Phi_y}{1+\eta} + \frac{y\xi}{1+\eta} F_1(U).
	\]
	The choice of the function spaces  is made precise later in Sections~\ref{s linear} and~\ref{s transverse}.
	
	The system \eqref{eq-hydro-evo}--\eqref{eq-hydro-bc-evo} inherits the symmetries of the Euler equations \eqref{eq-hydro-nondim}--\eqref{eq-hydro-bc-nondim}. As a consequence of the horizontal spatial reflection $z\mapsto-z$, the system \eqref{eq-hydro-evo}--\eqref{eq-hydro-bc-evo} is reversible with reversibility symmetry $R$ acting by
	\begin{equation}
	\label{eq-revers}
	R\begin{pmatrix} \eta\\ \omega\\\Phi \\ \xi \end{pmatrix} (x,y,z,t)= \begin{pmatrix}\eta\\ -\omega\\ \Phi\\ -\xi \end{pmatrix}(x,y,-z,t),
	\end{equation}
	which anti-commutes with $D$ and $F$ and commutes with $B$. The second horizontal spatial reflection $x\mapsto-x$, implies that the system \eqref{eq-hydro-evo}--\eqref{eq-hydro-bc-evo} possesses a reflection symmetry
	\begin{equation}
	\label{eq-reflec}
	S \begin{pmatrix}\eta\\ \omega\\ \Phi\\ \xi \end{pmatrix}(x,y,z,t)= \begin{pmatrix}\eta \\ \omega \\ -\Phi \\-\xi \end{pmatrix} (-x,y,z,t),
	\end{equation}
	which commutes with $D$, $F$, and $B$. There are in addition two continuous symmetries, which are the horizontal spatial translations in $x$ and $z$.

	\subsection{Two-dimensional steady periodic waves}
	\label{sect-2D-exist}
	
	Spatial dynamics also provides an efficient method for the study of the existence of two-dimensional steady waves of the Euler equations \eqref{eq-hydro-nondim}--\eqref{eq-hydro-bc-nondim}. This idea, which goes back to the work by Kirchg\"assner \cite{Ki82}, consists in writing the two-dimensional steady Euler equations as a dynamical system of the form
	\begin{equation}\label{e:2D}
	\frac{\dif U}{\dif x} = LU + R(U),
	\end{equation}
	in which $x$ is now the timelike variable, and $L$ and $R$ denote linear and nonlinear parts, respectively. A phase space $\mathcal X$ consisting of $y$-dependent functions is chosen such that the linear operator $L$ is closed with densely and compactly embedded domain $\mathcal Y\subset \mathcal X$. The first two boundary conditions in \eqref{eq-hydro-bc-nondim} are part of the definition of the domain $\mathcal Y$. There are several different such formulations of two-dimensional steady problem; see, for instance, \cite{Kirchgassner, HI} for two different formulations as a reversible dynamical system, and \cite{Groves04} for a formulation as a Hamiltonian system. 
	Two-dimensional steady water waves are bounded solutions of the dynamical system \eqref{e:2D} and can be found using tools from the theory of dynamical systems and bifurcation theory.
	
	In particular, periodic waves can be obtained by a direct application of the Lyapunov center theorem; see Theorem~\ref{app-thm-Lyapunov}. The key observation is that the purely imaginary eigenvalues of the operator $L$ are given by the real roots of the linear dispersion relation \eqref{eq-lindisrel}. As shown, for instance in \cite{Kirchgassner}, for any pair of parameters $(\alpha,\beta)$ in Region~I, the linear operator $L$ possesses precisely one pair of simple purely imaginary eigenvalues $\pm {\ii} k_*$, with $k_*$ the unique positive root of the linear dispersion relation \eqref{eq-lindisrel}. Similarly, in Region II there are two pairs of simple purely imaginary eigenvalue $\pm {\ii}k_{*,1}$ and $\pm {\ii}k_{*,2}$. Then, the reversibility of the dynamical system \eqref{e:2D} together with a direct check of the resolvent estimates \eqref{e:resolvent} allow to apply Theorem~\ref{app-thm-Lyapunov} and prove the result in Theorem~\ref{thm-main-II}(\ref{thm-main-II-exist}) for Region~I and the result in Theorem~\ref{thm-main-III}(\ref{thm-main-III-exist}) for Region~II.
	
	In addition, for our purposes we need to compute the first two terms of the expansion in $\eps$ of the two-dimensional periodic solutions. For $(\alpha,\beta)$ in Region~I, we write
	\[
	X=k_\eps x,\quad \eta_\eps(X) = \widetilde{\eta}_\eps(x),\quad
	\Phi_\eps(X,y)=\widetilde{\Phi}_\eps(x,y),
	\]
	so that $\widetilde{\eta}_\eps$ and $\widetilde{\Phi}_\eps$ are $2\pi$-periodic in $x$, and consider the expansions
	\begin{equation}\label{eq-2Dsol-exp}
	\begin{split}
	&k_\eps = k_* + \eps^2 k_2 + \O(\eps^3),\\
	&\widetilde{\eta}_\eps(x) = \eps \eta_1(x) + \eps^2 \eta_2(x) + \O(\eps^3),\\
	&\widetilde{\Phi}_\eps(x, y) = \eps \Phi_1(x,y) + \eps^2 \Phi_2(x,y) + \O(\eps^3),
	\end{split}
	\end{equation}
	where $k_*$ is the positive root of the linear dispersion relation \eqref{eq-lindisrel}. Substituting these expansions into the Euler equations \eqref{eq-hydro-nondim}--\eqref{eq-hydro-bc-nondim}, we obtain in Appendix~\ref{app-2D-calc} the following explicit formulas:
	\begin{equation} \label{e:k2}
	\begin{split}k_2 = \frac{k_*^3}{d(k_*)} \Big( 
	&
	\left(9\alpha \beta +16\right)k_*  
	- 12\alpha \beta k_*\cosh(2k_*)  + 3\alpha \beta k_* \cosh(4k_*)\\
	&
	- 8\alpha(2c(k_*) -1) \sinh(2k_*)-4\alpha (c(k_*)+2)\sinh(4k_*)
	\Big),
	\end{split}
	\end{equation}
	and
	\begin{equation}\label{e:expan}
	\begin{split}
	&\eta_1(x)=\sinh(k_*)\cos(x),\quad
	\Phi_1(x,y) = \cosh(k_*y)\sin(x), \\
	&\eta_2(x) = \frac{k_*}{4}\left(c(k_*) + 1 \right) \sinh(2k_*)\cos(2x) - \frac{k_*^2}{4\alpha}, \\
	&\Phi_2(x,y) =  \frac{k_*}{4}\left(c(k_*) \cosh(2k_*y)+2\sinh(k_*)y\sinh(k_*y)\right) \sin(2x),
	\end{split}
	\end{equation}
	where
	\begin{equation}\label{e:ck*}
	\begin{split}
	c(k_*)&=-1- \frac{k_*(\cosh(2k_*)+2)}{\mathcal D(2k_*)},\\
	d(k_*)&=32\alpha\left(2\beta k_*(\cosh(2k_*)-1) +  2k_* -\sinh(2k_*) \right),
	\end{split}
	\end{equation}
	and $\mathcal D(k)$ is the linear dispersion relation \eqref{eq-lindisrel}.
	In addition, the function $\widetilde{\eta}_\eps$ is even in $x$, whereas $\widetilde{\Phi}_\eps$ is an odd function.  
	
	For each $\eps \in (-\eps_0, \eps_0)$, the solution $({\eta}_\eps,{\Phi}_\eps)$ of the Euler equations \eqref{eq-hydro-nondim}--\eqref{eq-hydro-bc-nondim} provides a solution
	\begin{equation}\label{e:Ueps}
	U_\eps (x,y)= (\widetilde{\eta}_\eps(x), 0, \widetilde{\Phi}_\eps(x, y), 0)^{\TT}
	\end{equation}
	of the dynamical system~\eqref{eq-hydro-evo}--\eqref{eq-hydro-bc-evo} for $k=k_\eps$, hence satisfying
	\begin{equation}\label{e:Ueps0}
	\begin{cases} F(U_\eps) = 0,\\
	\widetilde{\Phi}_{\eps y} = B(U_\eps), & \text{on $y = 0,1$}.
	\end{cases}
	\end{equation}
	In addition, the above parity properties of  $\widetilde{\eta}_\eps$ and $\widetilde{\Phi}_\eps$ imply that $SU_\eps=U_\eps$ where $S$ is the reflection symmetry given in \eqref{eq-reflec}.

	\section{Analysis of the linear operator}
	\label{s linear}
	
	For fixed $(\alpha,\beta)$ in Region~I, we denote by $L_\eps$ the linear operator which appears  in the linearization of the dynamical sytem \eqref{eq-hydro-evo}--\eqref{eq-hydro-bc-evo} at the periodic wave $U_\eps$ for $k=k_\eps$. We prove the properties of $L_\eps$ needed for the transverse dynamics analysis in Section~\ref{s transverse}.
	
	For notational simplicity we remove the tilde from \eqref{e:Ueps0} and write from now on ${\eta}_\eps$ and  ${\Phi}_\eps$ instead of $\widetilde{\eta}_\eps$ and $\widetilde{\Phi}_\eps$, respectively.

	\subsection{The linear operator $\boldsymbol{L_\eps}$}
	\label{s Leps}
	
	A direct computation of the differential of $F$ at the periodic wave $U_\eps$ gives the following explicit formulas for $L_\eps U\coloneqq  \dif F[U_\eps] U$,
	\begin{equation}\label{eq-linop}
	L_\eps U
	=
	\begin{pmatrix} \omega/\beta + H_1(\omega, \xi) \\ \alpha \eta - \beta k_\eps^2 \eta_{xx} - k_\eps \Phi_x|_{y=1} + H_2(\eta, \Phi)\\ \xi + H_3(\omega, \xi) \\ -k_\eps^2 \Phi_{xx} - \Phi_{yy} + H_4(\eta, \Phi)\end{pmatrix},\quad
	U= \begin{pmatrix} \eta\\\omega\\ \Phi \\ \xi \end{pmatrix},
	\end{equation}
	where
	\[\begin{split}&H_1(\omega, \xi) = \frac{(1 + k^2_\eps\eta_{\eps x}^2)^{1/2}}{\beta} \left( \omega + \frac{1}{1+\eta_\eps} \int_0^1 y \Phi_{\eps y} \xi \dif y \right) - \frac{\omega}{\beta},\\
	&H_2(\eta, \Phi) = \beta k^2_\eps \eta_{xx} - \beta k^2_\eps \left[\frac{\eta_x}{(1+k^2_\eps \eta_{\eps x}^2)^{3/2}}\right]_x \\
	& \quad + \int_0^1 \Bigg\{k^2_\eps\Phi_{\eps x}\Phi_x - \frac{\Phi_{\eps y}\Phi_y}{(1+\eta_\eps)^2} + \frac{\Phi_{\eps y}^2\eta}{(1+\eta_\eps)^3}-k^2_\eps\frac{y^2\eta_{\eps x}^2\Phi_{\eps y}\Phi_y}{(1+\eta_\eps)^2}\\
	& \quad \qquad -k^2_\eps\frac{y^2\eta_{\eps x}\Phi_{\eps y}^2 \eta_x}{(1+\eta_\eps)^2}+ k_\eps\frac{y^2 \eta_{\eps x}\Phi_{\eps y}^2 \eta}{(1+\eta_\eps)^3}\\
	&\quad \qquad + k^2_\eps \left[y\Phi_{\eps y}\Phi_x + y\Phi_{\eps x}\Phi_y - \frac{2y^2\eta_{\eps x}\Phi_{\eps y}\Phi_y}{1+\eta_\eps}-\frac{y^2 \Phi_{\eps y}^2 \eta_x}{1+\eta_\eps} + \frac{y^2 \Phi_{\eps y}^2 \eta_{\eps x}\eta}{(1+\eta_\eps)^2}\right]_x\Bigg\} \dif y,\\
	\end{split}\]
	\[\begin{split}
	&H_3(\omega, \xi) = -\frac{\eta_\eps \xi}{1+\eta_\eps} + \frac{1}{1+\eta_\eps}\left(H_1(\omega, \xi)+\frac{\omega}{\beta}\right)y \Phi_{\eps y}, \hspace{5cm}\\
	&H_4(\eta, \Phi) = k^2_\eps\Big[-\eta_\eps \Phi_x - \Phi_{\eps x}\eta + y \Phi_{\eps y}\eta_x + y\eta_{\eps x}\Phi_y\Big]_x + \Big[k_\eps y\eta_x + B_{l\eps}(\eta, \Phi)\Big]_y.
	\end{split}\]
	To this expression of $L_\eps U$ we add the linear boundary conditions obtained by taking the differential of $B$ at $U_\eps$,
	\begin{equation}\label{e:linbc}
	\Phi_y=  B_{l\eps}(U) \coloneqq \dif B[U_\eps] U=0,\quad \text{on $y = 0,1$},
	\end{equation}
	where
	\[\begin{split}B_{l\eps}(U) &= k_\eps y(-\eta_x + k_\eps\eta_{\eps x} \Phi_x + k_\eps \Phi_{\eps x}\eta_x) \\
	& \quad + \frac{\eta_\eps \Phi_y}{1+\eta_\eps} + \frac{\Phi_{\eps y} \eta}{(1+\eta_\eps)^2} + k^2_\eps\frac{y^2 \eta_{\eps x}^2 \Phi_{\eps y} \eta}{(1+\eta_\eps)^2} - k^2_\eps\frac{y^2\eta_{\eps x}^2 \Phi_y}{1+\eta_\eps} - 2k^2_\eps \frac{y^2 \eta_{\eps x}\Phi_{\eps y} \eta_x}{1+\eta_\eps}.
	\end{split}\]
	Notice that $B_{l\eps}(U)$ only depends on the components $\eta$ and $\Phi$ of $U$. We will sometimes write $B_{l\eps}(\eta, \Phi)$ for convenience.

	For $s\geq0$, we define the Hilbert space
	\begin{equation}\label{e:Hs}
	\mathcal X^s = H^{s+1}_{\text{per}}(\mathbb S) \times H^s_{\text{per}}(\mathbb S) \times H^{s+1}_{\text{per}}(\Sigma) \times H^s_{\text{per}}(\Sigma),
	\end{equation}
	where $\mathbb{S} = (0, 2\pi)$, $\Sigma = \mathbb{S} \times (0,1)$, and
	\[\begin{split} &H^{s}_{\per}(\mathbb S) = \{u \in H^{s}_{\loc}(\R) \, : \, u(x+2\pi) = u(x), \ x \in \R\},\\
	&H^{s}_{\per}(\Sigma) = \{u \in H^{s}_{\loc}(\R \times (0,1)) \, : \, u(x+2\pi, y) = u(x,y), \ y \in (0,1), \ x \in \R\}.
	\end{split}\]
	The action of the operator $L_\eps$ is taken in $\mathcal X^0$ with domain of definition
	\[\mathcal Y^1_\eps = \{U=(\eta, \omega, \Phi, \xi)^{\TT} \in \mathcal X^1 \, : \, \Phi_y = B_{l\eps}(\eta, \Phi) \ \text{on $y=0,1$}\},\]
	chosen to include the boundary conditions. Then $L_\eps$ is well-defined and closed in $\mathcal X^0$, and its domain $\mathcal Y^1_\eps$  is compactly embedded in $\mathcal X^0$. The latter property implies that the operator $L_\eps$ has pure point spectrum consisting of isolated eigenvalues with finite algebraic multiplicity. As a consequence of the reflection symmetry $S$ given in \eqref{eq-reflec}, which commutes with $F$ and leaves invariant $U_\eps$, the subspaces
	\begin{equation}\label{e:X0pm}
	\mathcal X^0_+=\{U\in \mathcal X^0 \,:\, SU=U\},\quad
	\mathcal X^0_-=\{U\in \mathcal X^0 \,:\, SU=-U\},
	\end{equation}
	are invariant under the action of $L_\eps$.
	
	One inconvenience of this functional-analytic setting is that the domain of definition $\mathcal Y^1_\eps$ of the linear operator $L_\eps$ depends on $\eps$. This difficulty is well-known and can be handled using an appropriate change of variables first introduced for the three-dimensional steady nonlinear Euler equations in \cite{GM}. Here, we proceed as in \cite{GHS} and replace $\Phi$ by $\Upsilon =\Phi + \chi_y$, where $\chi$ is the unique solution of the elliptic problem
	\[
	\begin{aligned}-k_\eps^2\chi_{xx}-\chi_{yy}&=B_{l\eps}(U)&  &\text{in $\Sigma$},\\
	\chi&=0& &\text{on $y = 0,1$,} \end{aligned}
	\]
	so that $\Upsilon$ satisfies the boundary conditions $\Upsilon_y=0$ on $y=0,1$ which do not depend on $\eps$. The linear mapping defined by $G_\eps(\eta,\omega,\Phi,\xi)^{\TT} = (\eta,\omega,\Upsilon,\xi)^{\TT}$, is a linear isomorphism in both $\mathcal X^0$ and $\mathcal X^1$, it depends smoothly on $\eps$ and the same is true for its inverse $G_\eps^{-1}$. Setting $\widetilde L_\eps= G_\eps L_\eps G_\eps^{-1}$ the operator $\widetilde L_\eps$ acts in $\mathcal X^0$ with domain of definition
	\[\mathcal Y^1 = \{U=(\eta, \omega, \Upsilon, \xi)^{\TT} \in \mathcal X^1 \, : \,\Upsilon_y = 0 \ \text{on $y=0,1$}\},\]
	which does not depend on $\eps$ anymore. While $\widetilde L_\eps$ allows us to rigorously apply general results for linear operators, it is more convenient to use $L_\eps$ for explicit computations.

	\subsection{Spectral properties of $\boldsymbol{L_0}$}
	\label{s:L0}
	
	The unperturbed operator $L_0$ obtained for $\eps=0$ is a differential operator with constant coefficients. Therefore, eigenvalues, eigenfunctions, and generalized eigenfunctions can be explicitly computed using Fourier series in the variable $x$. In particular, for purely imaginary values ${\ii}\ell$ with $\ell\in\R$ the eigenvalue problem $(L_0 - {\ii}\ell\, \mathbb I)U = 0$ possesses nontrivial solutions in the $n^{\text{th}}$ Fourier mode  if and only if
	\[
	(\alpha + \beta \sigma^2)\sigma \sinh \sigma
	-n^2k_*^2 \cosh \sigma= 0 \quad\text{with}\quad \sigma^2=n^2k_*^2+\ell^2.
	\]
	For fixed $(\alpha,\beta)$ in Region~I, this equality holds if and only if $\ell = 0$ and $n\in\{0, \pm 1\}$; see also \cite{Groves}. Consequently, $0$ is the only purely imaginary eigenvalue of $L_0$ and it has geometric multiplicity three. The associated eigenvectors are given by the explicit formulas:
	\begin{equation}\label{e:eig0}
	\begin{split} &\zeta_0 = \begin{pmatrix} 0 \\ 0 \\ 1 \\ 0 \end{pmatrix}, \quad \zeta_-= \begin{pmatrix} -\sinh(k_*)\sin(x) \\ 0 \\ \cosh(k_*y)\cos(x) \\ 0 \end{pmatrix}, \quad \zeta_+ = \begin{pmatrix} \sinh(k_*)\cos(x) \\ 0 \\ \cosh(k_*y) \sin(x) \\ 0 \end{pmatrix}.
	\end{split}
	\end{equation}
	Associated to each eigenvector there is a Jordan chain of length two, so that the algebraic multiplicity of the eigenvalue $0$ is six. The generalized eigenvectors associated to $\zeta_0, \zeta_-$ and $\zeta_+$ are given by, respectively,
	\begin{equation}\label{e:geneig0}
	\psi_0 = \begin{pmatrix} 0 \\ 0 \\ 0 \\ 1 \end{pmatrix}, \quad \psi_- = \begin{pmatrix}0\\ -\beta \sinh(k_*) \sin(x) \\ 0 \\ \cosh(k_*y)\cos(x)\end{pmatrix}, \quad \psi_+=\begin{pmatrix} 0 \\ \beta\sinh(k_*)\cos(x) \\ 0 \\ \cosh(k_*y)\sin(x)\end{pmatrix}.
	\end{equation}
	Notice that the reflection symmetry $S$ given in \eqref{eq-reflec} acts on these eigenvectors as follows: 
	\[\begin{aligned}
	&S\zeta_0 = -\zeta_0,& \quad &S\zeta_-=-\zeta_-,&\quad &S\zeta_+=\zeta_+,&\\
	&S\psi_0= - \psi_0,&\quad &S\psi_-=-\psi_-,& \quad &S\psi_+ = \psi_+.&
	\end{aligned}\]
	These formulas are consistent with the ones already found in \cite{Groves}. The remaining eigenvalues of $L_0$ are bounded away from the imaginary axis.
	
	\subsection{Main result}
	
	We summarize in the next theorem the properties of the linear operator $L_\eps$ needed for our transverse dynamics analysis. The same properties hold for the operator $\widetilde L_\eps$.
	
	\begin{thm}[Linear operator]\label{lem-linop-spectral}
		There exist positive constants $\eps_1$, $C_1$, and $\ell_1$, such that for each $\eps \in (-\eps_1, \eps_1)$ the following properties hold.
		\begin{enumerate}[(i)]
			\item \label{lem-linop-spectral-i}
			The linear operator $L_\eps$ acting in $\mathcal X^0$ with domain $\mathcal Y^1_\eps$ has an eigenvalue $0$ with algebraic multiplicity four, and two simple purely imaginary eigenvalues $\pm {\ii} \ell_\eps$ with $\ell_\eps>0$ and $\ell_0=0$. Any other purely imaginary value ${\ii}\ell\in{\ii}\R \setminus \{0, \pm {\ii}\ell_\eps\}$ belongs to the resolvent set of $L_\eps$.
			\item \label{lem-linop-spectral-ii} The restriction of $L_\eps$ to the invariant subspace $\mathcal X^0_+$ has the two simple purely imaginary eigenvalues $\pm {\ii} \ell_\eps$ and any other value ${\ii}\ell\in{\ii}\R \setminus \{\pm{\ii}\ell_\eps\}$ belongs to the resolvent set.
			\item \label{lem-linop-spectral-iii} The inequality 
			\[\left\|(L_\eps-{\ii}\ell \,\mathbb I)^{-1}\right\|_{\mathcal L(X^0)} \leq \frac{C_1}{|\ell|},\]
			holds for each real number $\ell$ with $|\ell|>\ell_1$.
		\end{enumerate}
	\end{thm}
	
	\begin{proof}
		We rely on the properties of the operator $L_0$ and perturbation arguments for $\eps$ sufficiently small. The operators $\widetilde L_\eps$ and $\widetilde L_0$ having the same domain of definition $\mathcal Y^1$, standard perturbation arguments show that $\widetilde L_\eps$ is a small relatively bounded perturbation of $\widetilde L_0$ for $\eps$ sufficiently small.
		The result in item~({\ref{lem-linop-spectral-iii}}) is an immediate consequence of this property. Indeed, for $\eps=0$ the inequality from~(\ref{lem-linop-spectral-iii}) is given in \cite{Groves}, which implies that a similar inequality holds for $\widetilde L_0$, with possibly different values $C_1$ and $\ell_1$. The operator $\widetilde L_\eps$ being a relatively bounded perturbation of $\widetilde L_0$ for sufficiently small $\eps$, from the inequality for $\widetilde L_0$ we obtain that item~(\ref{lem-linop-spectral-iii}) holds for $\widetilde L_\eps$, and then for $L_\eps$. It remains to prove items~(\ref{lem-linop-spectral-i}) and~(\ref{lem-linop-spectral-ii}). This is the main part of the proof of the theorem.
		
		\paragraph{Spectral decomposition.}  The results in Section~\ref{s:L0} show that the spectrum $\sigma(L_0)$ of the linear operator $L_0$ satisfies 
		\[
		\sigma(L_0)=\{0\}\cup\sigma_1(L_0),\quad
		\sigma_1(L_0)\subset\{\lambda\in\mathbb C\,:\, |{\RE}\, \lambda|> d_1\},
		\]
		for some $d_1>0$, where $0$ is an eigenvalue with algebraic multiplicity six and geometric multiplicity three, and the same is true for the linear operator $\widetilde L_0$. The six-dimensional spectral subspace $\mathcal E_0$ associated to the eigenvalue $0$ of $L_0$ is spanned by the eigenvectors $\zeta_0,\ \zeta_{\pm}$ given in \eqref{e:eig0} and generalized eigenvectors $\psi_0,\ \psi_{\pm}$ given in \eqref{e:geneig0}. For $\eps \neq 0$ sufficiently small, $\widetilde L_\eps$ is a small relatively bounded perturbation of $\widetilde L_0$. Consequently, there exists a neighborhood $V_0\subset\mathbb C$ of the origin such that 
		\[
		V_0\subset \{\lambda\in\mathbb C\,:\, |{\RE}\,\lambda|< d_1/4\}
		\]
		and
		\[
		\sigma(\widetilde L_\eps)=\sigma_{0}(\widetilde L_\eps)\cup\sigma_{1}(\widetilde L_\eps),\quad
		\sigma_{0}(\widetilde L_\eps)\subset V_0,\quad  \sigma_{1}(\widetilde L_\eps)\subset\{\lambda\in\mathbb C\,:\, |{\RE}\,\lambda|> d_1/2\},
		\]
		for sufficiently small $\eps$,  where the spectral subspace associated to $\sigma_{0}(\widetilde L_\eps)$ is six-dimensional, and the same is true for $L_\eps$. Moreover, for the operator $L_\eps$, there exists a basis $ \{\zeta_0(\eps), \zeta_{\pm}(\eps),$ $\psi_0(\eps), \psi_{\pm}(\eps)\}$ of the six-dimensional spectral subspace $\mathcal E_\eps$ associated to  $\sigma_{0}(L_\eps)$ which is the smooth continuation, for sufficiently small $\eps$, of the basis $\{\zeta_0,\ \zeta_{\pm},\ \psi_0,\ \psi_{\pm}\}$ of the six-dimensional spectral subspace $\mathcal E_0$ associated to the eigenvalue $0$ of $L_0$. The two bases share the symmetry properties,
		\[\begin{aligned}
		&S\zeta_0(\eps) = -\zeta_0(\eps),& \quad &S\zeta_-(\eps)=-\zeta_-(\eps),&\quad &S\zeta_+(\eps)=\zeta_+(\eps),&\\
		&S\psi_0(\eps)= -\psi_0(\eps),&\quad &S\psi_-(\eps)=-\psi_-(\eps),& \quad &S\psi_+(\eps) = \psi_+(\eps).&
		\end{aligned}\]
		Thus, we have the decomposition $\mathcal E_\eps = \mathcal E_{\eps,+} \oplus \mathcal E_{\eps,-}$ with
		\[\begin{split}
		&\mathcal E_{\eps, +}=\{U \in \mathcal E_{\eps} \, : \, RU = U\}\hspace{0.3cm}=\text{span} \{\zeta_+(\eps), \psi_+(\eps)\},\\
		&\mathcal E_{\eps,-}=\{U \in \mathcal E_\eps \, : \, RU = -U\} =\text{span}\{\zeta_0(\eps), \zeta_-(\eps), \psi_0(\eps), \psi_-(\eps)\}.
		\end{split}\]
		These spaces $\mathcal E_{\eps, \pm}$ are invariant under the action of $L_\eps$.
		
		Purely imaginary eigenvalues of $L_\eps$ necessarily belong to the neighborhood $V_0$ of $0$. Therefore, they are detemined by the action of $L_\eps$ on the spectral subspace $\mathcal E_\eps$. This action is represented by a $6\times 6$ matrix. The decomposition  $\mathcal E_\eps = \mathcal E_{\eps,+} \oplus \mathcal E_{\eps,-}$ above, implies that we can further decompose the action of  $L_\eps$ by restricting to the invariant subspaces $\mathcal E_{\eps, \pm}$. In other words, the  $6\times 6$ matrix is a block matrix with a $2 \times 2$ block  representing the action of $L_\eps$ on $\mathcal E_{\eps, +}$ and a $4\times 4$ block representing the action of $L_\eps$ on $\mathcal E_{\eps, -}$. Our task is to determine the eigenvalues of these two matrices. This will prove the result in part~(\ref{lem-linop-spectral-i}) of the theorem. For the restriction of the linear operator $L_\eps$ to the invariant subspace $\mathcal X^0_+$ in part~(\ref{lem-linop-spectral-ii}) of the theorem, it is enough to consider the eigenvalues of the $2 \times 2$ matrix.
		
		\paragraph{Eigenvalues of the  $\boldsymbol{4\times 4}$ matrix.} It turns out that a basis of the subspace $\mathcal E_{\eps,-}$ can be explicitly obtained using the symmetries of the Euler equations. First, the Euler equations \eqref{eq-hydro-nondim}--\eqref{eq-hydro-bc-nondim} are invariant under the transformation $\phi \mapsto \phi + C$ for any real constant $C$. This implies that the dynamical system  \eqref{eq-hydro-evo}--\eqref{eq-hydro-bc-evo} is invariant under the transformation $U\mapsto U+\zeta_0$ where $\zeta_0=(0,0,1,0)^{\text{T}}$. Consequently, $\zeta_0$ belongs to the kernel of $L_\eps$ and since $S\zeta_0=-\zeta_0$ it belongs to $\mathcal E_-$. We choose $\zeta_0(\eps)=\zeta_0$ and then a direct computation gives the generalized eigenvector
		\[\psi_0(\eps) = \begin{pmatrix} 0 \\ -\int_0^1 y \Phi_{\eps y} \dif y \\ 0 \\ 1+\eta_\eps\end{pmatrix},\]
		satisfying $L_\eps\psi_0(\eps)=\zeta_0$ and  $S\psi_0(\eps)=-\psi_0(\eps)$.
		
		Next, the invariance of the Euler equations under horizontal spatial translations in $x$ implies that the derivative $U_{\eps x}=(\eta_{\eps x}, 0, \Phi_{\eps x}, 0)^{\text{T}}$ of the periodic wave belongs to the kernel of $L_\eps$. Since $SU_{\eps x}=-U_{\eps x}$, the vector $U_{\eps x}$ belongs to $\mathcal E_{\eps,-}$. From the expansions~\eqref{eq-2Dsol-exp}, we find that $U_{\eps x}=\eps\zeta_-+\O(\eps^2)$. This gives a second vector $\zeta_-(\eps) = \eps^{-1}U_{\eps x}$ which belongs to the kernel of $L_\eps$, and also to the invariant subspace $\mathcal E_-$, with the property that $\zeta_-(\eps)\to \zeta_-$ as $\eps \to 0$. The corresponding generalized eigenvector is given by 
		\[\psi_-(\eps) = \frac{1}{\eps}\begin{pmatrix} 0 \\ \dfrac{\eta_{\eps x}\beta}{(1+k_\eps^2 \eta_{\eps x}^2)^{1/2}} - \int_0^1 y \Phi_{\eps y} \left(\Phi_{\eps x}-\dfrac{\eta_{\eps x}y\Phi_{\eps y}}{1+\eta_{\eps}}\right) \dif y \\ 0 \\ (1+\eta_\eps)\left(\Phi_{\eps x}-\dfrac{\eta_{\eps x}y\Phi_{\eps y}}{1+\eta_\eps}\right)\end{pmatrix}.\]
		
		The above shows that there is a basis $\{\zeta_0(\eps), \psi_0(\eps), \zeta_-(\eps), \psi_-(\eps)\}$ for $\mathcal E_{\eps,-}$ satisfying
		\[
		L_\eps\zeta_0(\eps) = 0,\quad L_\eps \psi_0(\eps) = \zeta_0(\eps),\quad L_\eps \zeta_-(\eps) = 0,\quad L_\eps \psi_-(\eps) = \zeta_-(\eps).
		\]
		Thus, $0$ is the only eigenvalue of the $4\times 4$ matrix representing the action of $L_\eps$ onto $\mathcal E_{\eps,-}$ and it has geometric multiplicity two and algebraic multiplicity four.

		\paragraph{Eigenvalues of the  $\boldsymbol{2\times 2}$ matrix.} We consider a basis $\{\zeta_+(\eps), \psi_+(\eps)\}$ of the subspace $\mathcal E_{\eps,+}$ which is the smooth continuation of the basis $\{\zeta_+,\psi_+\}$ of $\mathcal E_{0,+}$, and denote by $\mathcal M(\eps)$ the $2 \times 2$ matrix representing the action of $L_\eps$ on this basis. At $\eps=0$, we have that $L_0\zeta_+=0$ and $L_0\psi_+=\zeta_+$, which implies that 
		\[
		\mathcal M(0) = \begin{pmatrix} 0 & 1 \\ 0 & 0 \end{pmatrix}.
		\]
		
		For $\eps \neq 0$, we write 
		\[
		\mathcal M(\eps) = \begin{pmatrix} m_{11}(\eps) & 1+m_{12}(\eps)\\ m_{21}(\eps) & m_{22}(\eps) \end{pmatrix}.
		\]
		The invariance of the Euler equations under horizontal spatial translations in $x$, implies that the periodic waves translated by a half-period $\pi$ are also periodic solutions. Comparing their expansions in $\eps$ with the ones of $(\eta_\eps, \Phi_\eps)$ we conclude that
		\[
		\eta_\eps(x)=\eta_{-\eps}(x+\pi), \quad \Phi_\eps(x,y)=\Phi_{-\eps}(x+\pi,y).
		\]
		Since the $2\times2$ matrices corresponding to these solutions are the same this implies that $\mathcal M(\eps) = \mathcal M(-\eps)$, and as a consequence, we have the expansion $m_{ij}(\eps)=m_{ij}^{(2)}\eps^2 + \mathcal O(\eps^4)$, for $\eps$ sufficiently small.
		
		Next, the reversibility of $L_\eps$ implies that the spectrum of $L_\eps$ is symmetric with respect to the origin in the complex plane. Moreover, because $L_\eps$ is a real operator, its spectrum is also symmetric with respect to the real line. These observations combined imply that the two eigenvalues of $\mathcal M(\eps)$ are either both real or both purely imaginary, and their sum is equal to $0$. Consequently, $m_{11}(\eps)=-m_{22}(\eps)$. Further, the product of these two eigenvalues is equal to the determinant of $\mathcal M(\eps)$. Therefore, the eigenvalues are both real if $\det \mathcal M(\eps) < 0$ and both purely imaginary if $\det \mathcal M(\eps)>0$. We have
		\[
		\det \mathcal M(\eps) = -m_{11}^2(\eps)-m_{21}(\eps)(1+m_{12}(\eps)) = -m_{21}^{(2)}\eps^2 + \mathcal O(\eps^4),\]
		so the result in theorem holds provided $m_{21}^{(2)}<0$.
		
		The final step is the computation of the sign of $m_{21}^{(2)}$. We prove in Appendix~\ref{app-m212} that
		\begin{equation}\label{eq-m212}
		m_{21}^{(2)} = - 4k_* \cdot \frac{4\beta k_* \sinh^2(k_*)+ 2k_* -\sinh(2k_*) }{4\beta k_* \sinh^2(k_*)+2k_* + \sinh(2k_*)} \cdot k_2,
		\end{equation}
		where $k_2$ is the coefficient in the expansion of the wavenumber $k_\eps$ given in \eqref{e:k2}. Replacing the formula for $k_2$ we obtain
		\[
		m_{21}^{(2)} = \frac{k_*^4}{8\alpha}\cdot \frac1{4\beta k_* \sinh^2(k_*)+2k_* + \sinh(2k_*)}\cdot \widetilde m_{21}^{(2)},
		\]
		where
		\begin{equation}\label{eq-coeff-m212}
		\begin{split}
		\widetilde m_{21}^{(2)} = & 
		-\left(9\alpha \beta +16\right)k_*  
		+ 12\alpha \beta k_*\cosh(2k_*)  - 3\alpha \beta k_* \cosh(4k_*)\\[0,5ex]
		&
		+8\alpha(2c(k_*) -1) \sinh(2k_*)+4\alpha (c(k_*)+2)\sinh(4k_*),
		\end{split}
		\end{equation}
		with $c(k_*)<-1$ given in \eqref{e:ck*}. Clearly, $m_{21}^{(2)}$ and $\widetilde m_{21}^{(2)}$ have the same sign. Proposition~\ref{prop-coeff-sign} in Appendix~\ref{app-m212-sign} shows that $\widetilde m_{21}^{(2)}<0$ for $(\alpha,\beta)$ in Region~I. This completes the proof of the theorem.
	\end{proof}
	
	\section{Transverse dynamics}
	\label{s transverse}
	
	We show that the two-dimensional periodic waves in Theorem~\ref{thm-main-II}(\ref{thm-main-II-exist}) are linearly transversely unstable for $\eps$ sufficiently small, and then discuss the induced dimension-breaking bifurcation. These two results prove the parts (\ref{thm-main-II-unstable}) and  (\ref{thm-main-II-dimbreak}) of Theorem~\ref{thm-main-II}. 
	
	Throughout this section, we consider a two-dimensional periodic wave $U_\eps$ such that the associated linearized of operator $L_\eps$ studied in Section~\ref{s linear} possesses two simple purely imaginary eigenvalues $\pm{\ii}\ell_\eps$ as in Theorem~\ref{lem-linop-spectral}, hence by fixing $\eps\in(-\eps_1,\eps_1)$. 
	
	\subsection{Transverse linear instability}
	\label{sect-transverse}
	
	Linearizing the system \eqref{eq-hydro-evo}--\eqref{eq-hydro-bc-evo} at $U_\eps$ we obtain the linear system
	\begin{equation}\label{eq-transverse}
	\frac{\dif U}{\dif z}= DU_t + \dif F[U_\eps]U,
	\end{equation}
	with boundary conditions
	\begin{equation}\label{eq-transverse-bc}
	\Phi_y = y\eta_t + B_{l\eps}(U) \quad \text{on $y=0,1$}.
	\end{equation}
	The periodic wave $U_\eps$ is transversely linearly unstable if the linear equation has a solution of the form $\exp(\sigma t) U_\sigma(z)$ with $\RE \sigma > 0$ and $U_\sigma \in C^1_{\text{b}}(\R, \mathcal X^0) \cap C_{\text{b}}(\R, \mathcal X^1)$. 
	For the construction of such a function, we closely follow the approach developed in \cite{GSW} where the authors studied the transverse instability of solitary waves for the Euler equations. The only difference is that the functions were localized in $x\in\R$ in \cite{GSW}, whereas here they are periodic in $x$. We use the following general result from \cite{GSW}, which we have slightly modified; see Remark~\ref{rem-Godey}.
	
	\begin{thm}[Theorem 1.3~\cite{GSW}] \label{thm-Godey} Consider real Banach spaces $\mathcal X$, $\mathcal Z$, $\mathcal Z_i$, $i=1,2$, and  a partial differential equation of the form
		\begin{equation}\label{eq-Godey}
		\frac{\dif U}{\dif z} = D_1 U_t + D_2 U_{tt}+ LU.
		\end{equation}
		Assume that the following properties hold: 
		\begin{enumerate}[(i)]
			\item $\mathcal Z\subset \mathcal Z_i \subset \mathcal X$, $i=1,2$,  with continuous and dense embeddings;
			\item $L,D_1$, and $D_2$ are closed linear operators in $\mathcal X$ with domains $\mathcal Z, \mathcal Z_1$, and $\mathcal Z_2$, respectively;
			\item the spectrum of $L$ contains a pair of isolated purely imaginary eigenvalues $\pm {\ii} \ell_*$ with odd multiplicity;
			\item there exists an involution $R \in \mathcal L(\mathcal X)$ which anticommutes with $L$ and $D_i$, $i=1,2$, i.e., the equation \eqref{eq-Godey} is reversible.
		\end{enumerate}
		Then, for each sufficiently small $\sigma>0$, equation~\eqref{eq-Godey} has a solution of the form $\exp(\sigma t) U_\sigma(z)$ with $\RE \sigma > 0$ and $U_\sigma \in C^1(\R, \mathcal X) \cap C(\R, \mathcal Z)$ a periodic function.  
	\end{thm}
	
	\begin{rem}\label{rem-Godey}
		Theorem 1.3~\cite{GSW} assumes that the linear operators $L, D_1,$ and $D_2$ have the same domain of definition, while we in Theorem~\ref{thm-Godey} allow for different domains, just like in Theorem 2.1 of~\cite{Godey}. This is needed since the operators $D_1$ and $D_2$ in our application (and in fact also in~\cite{GSW}) have different domains than $L$. Note in particular that the hypotheses imply that $D_1$ and $D_2$ are relatively bounded perturbations of $L$ by Remarks 1.4 and 1.5, Chapter 4.1.1~\cite{Kato}.
	\end{rem}
	
	This general result does not directly apply to the system~\eqref{eq-transverse}--\eqref{eq-transverse-bc} because the boundary condition \eqref{eq-transverse-bc} contains the extra term $y\eta_t$ which involves a derivative with respect to $t$. We proceed as in \cite{GSW} and eliminate this term by an appropriate change of variables, similar to the one used for $L_\eps$ in Section~\ref{s linear}.
	
	We replace the variable $\Phi$ in $U$ by a new variable $\Theta= \Phi + {\theta}_{yt}$ where ${\theta}$ is the unique solution of the elliptic boundary value problem
	\[
	\begin{aligned}
	-k_{\eps}^2{\theta}_{xx}-\theta_{yy} + B_{l\eps}(0,{\theta}_y) &= y\eta & &\text{in $\Sigma$}, \\
	{\theta} & = 0 & & \text{on $y = 0,1$},
	\end{aligned}
	\]
	where we set $B_{l\eps}(\eta,\Phi)=B_{l\eps}(U)$ because $B_{l\eps}(U)$ only depends on $\eta$ and $\Phi$. In the boundary value problem for $\theta$, we regard $t$ as a parameter and assume analytic dependence on $t$. Then the mapping defined by ${Q}(\eta, \omega, \Phi, \xi)^{\TT}= (\eta, \omega, {\Theta}, \xi)^{\TT}$ is a linear isomorphism on both ${\mathcal X}^0$ and ${\mathcal X}^1$. The transformed linearized problem~\eqref{eq-transverse}--\eqref{eq-transverse-bc} for $V= (\eta, \omega, {\Theta}, \xi)^{\TT}$ is of the form
	\begin{equation} \label{eq-transverse2}
	\frac{\dif V}{\dif z} = D_1V_t + D_2 V_{tt} + L_\eps V,
	\end{equation}
	with boundary conditions
	\begin{equation}\label{eq-transverse2-bc}
	{\Theta}_y = B_{l\eps}(V) \quad \text{on $y=0,1$}.
	\end{equation}
	The two linear operators $D_1$ and $D_2$ are bounded in ${\mathcal X}^0$ and defined by
	\[D_1 \begin{pmatrix} \eta \\ \omega \\ {\Theta} \\ \xi \end{pmatrix} = \begin{pmatrix} 0 \\ {\Theta}|_{y=1} + k_\eps{\theta}_{xy}|_{y=1} - H_2(0, {\theta}_y)\\ \widehat{\theta}_y \\ -\eta -  k_\eps^2(-\eta_{\eps}\theta_{xy}+y\eta_{\eps x}\theta_{yy})_x \end{pmatrix}, \quad D_2\begin{pmatrix} \eta \\ \omega\\ {\Theta} \\ \xi \end{pmatrix} = \begin{pmatrix} 0 \\ -{\theta}_y|_{y=1} \\ 0 \\ 0 \end{pmatrix},\]
	where $\widehat{\theta}$ is the unique solution of elliptic boundary value problem
	\[\begin{aligned}
	-k_\eps^2 \widehat{\theta}_{xx}-\widehat{\theta}_{yy} + B_{l\eps}(0, \widehat{\theta}_y)  &= y \left(\frac{\omega}{\beta}+H_1(\omega, \xi)\right) & & \text{in $\Sigma$},
	\\
	\hat{\theta}  &= 0 & &\text{on $y=0,1$}.
	\end{aligned}\]
	We use the system \eqref{eq-transverse2}--\eqref{eq-transverse2-bc} and the result in Theorem~\ref{thm-Godey} to prove the transverse linear instability of the periodic wave $U_\eps$.
	
	\begin{proof}[Proof of Theorem~\ref{thm-main-II}(\ref{thm-main-II-unstable})] We apply Theorem~\ref{thm-Godey} to the equation~\eqref{eq-transverse2} with Hilbert spaces $\mathcal X=\mathcal Z_i ={\mathcal X}^0$, $\mathcal Z={\mathcal Y}^1_\eps$, $i=1,2$, operators $D_1,D_2$ defined as above, and $L_{\eps}$. Since $D_1$ and $D_2$ are bounded on $\mathcal X^0$, they are closed operators in $\mathcal X^0$. The first two hypotheses {\it (i)} and {\it (ii)} are satisfied.
		The spectral condition {\it (iii)} is verified by Theorem~\ref{lem-linop-spectral}(\ref{lem-linop-spectral-i}). The reverser is $R$ defined in Section~\ref{sect-spatial-formulation} and its anti-commutativity with $D_1,D_2$ and $L$ is preserved by the change of variables $Q$. Thus, equation~\eqref{eq-transverse2} is reversible. Theorem~\ref{thm-Godey} now gives the statement of Theorem~\ref{thm-main-II}(\ref{thm-main-II-unstable}).\end{proof}

	\subsection{Dimension-breaking bifurcation}
	\label{sect-dimbreak}
	
	We look for three-dimensional steady solutions of the system~\eqref{eq-hydro-evo}--\eqref{eq-hydro-bc-evo} which bifurcate from the transversely unstable periodic wave $U_\eps$. Taking
	\[U(x, y, z) = U_\eps(x, y) + \widetilde U (x, y, z), \quad \widetilde U =(\eta,\omega,\Phi,\xi)^{\TT},\]
	in \eqref{eq-hydro-evo}--\eqref{eq-hydro-bc-evo} we obtain the equation
	\begin{equation}\label{eq-dimbreak}
	\frac{\dif \widetilde U }{\dif z} = F(U_\eps + \widetilde U ),
	\end{equation}
	together with the boundary conditions
	\begin{equation}
	\label{eq-dimbreak-bc}
	\Phi_y = B(U_\eps+\widetilde U )-B(U_\eps), \quad \text{on $y=0,1$}.
	\end{equation}
	The mappings $F$ and $B$ are defined on an open neighborhood $M$ of $0 \in \mathcal X^1$ which is contained in the set
	\[
	\{(\eta, \omega, \Phi, \xi)^{\TT} \in \mathcal X^1 \, : \, |W(x)| < \beta, \eta(x) > -1 \ \text{for all $x \in \R$}\},
	\]
	and are analytic. The periodic wave $U_\eps$ belongs to $M$, for sufficiently small $\eps$, and we look for bounded solutions $\widetilde U$ such that $U_\eps+\widetilde U(z)\in M$, for all $z\in\R$. 
	
	General bifurcation results cannot be directly applied to this system because the boundary condition~\eqref{eq-dimbreak-bc} is nonlinear. We make a nonlinear change of variables which transforms these nonlinear boundary conditions into linear boundary conditions. Similarly to our previous changes of variables from Section~\ref{s Leps} and Section~\ref{sect-transverse}, we replace $\Phi$ by a new variable $\Theta = \Phi + \theta_y$ where $\theta$ is the unique solution of the elliptic boundary value problem 
	\[\begin{aligned}
	-k_\eps^2\theta_{xx}-\theta_{yy}+B_{l\eps}(0,\theta_y)&=B(U)& & \text{in $\Sigma$},\\
	\theta &= 0& &\text{on $y=0,1$}.
	\end{aligned}\]
	and define $Q(\eta, \omega, \Phi, \xi)^{\TT}=(\eta, \omega, \Theta, \xi)^{\TT}$. Using the method from \cite{GHS} (see also \cite{GSW}) one can show that $Q$ is a near-identity analytic diffeomorphism from a neighborhood $M_1$ of $0 \in \mathcal X^1$ onto possibly a different neighborhood $M_2$ of $0 \in \mathcal X^1$ and that for each $U \in M_1$, the linear operator $\dif Q[U]\colon \mathcal X^1 \to \mathcal X^1$ extends to an isomorphism $\widehat{\dif Q}[U]\colon \mathcal X^0 \to \mathcal X^0$ which depends analytically on $U$ and the same holds for the inverse $\widehat{\dif Q}[U]^{-1}$. Then the equation \eqref{eq-dimbreak} is transformed into
	\begin{equation}\label{eq-dimbreak2}
	\frac{\dif V}{\dif z} = L_\eps V + N(V),
	\end{equation}
	where
	\[
	N\coloneqq \widetilde{F}-L_\eps,\quad
	\widetilde{F}(V) = \widehat{\dif Q}[Q^{-1}(V)](F(U_\eps+Q^{-1}(V))),
	\]
	and the boundary condition~\eqref{eq-dimbreak-bc} becomes linear,
	\begin{equation*}
	\Theta_y = B_{l\eps}(V)\quad \text{on $y=0,1$}.
	\end{equation*}
	In particular, we recover the linear operator $L_\eps$ studied in Section~\ref{s linear}, and we can apply the Lyapunov center theorem to conclude.
	
	\begin{proof}[Proof of Theorem~\ref{thm-main-III}(\ref{thm-main-III-dimbreak})]
		The equation \eqref{eq-dimbreak2} is a dynamical system in the phase space $\mathcal X^0$ with vector field defined in a neighborhood of $0$ in $\mathcal Y^1_\eps$. Because the change of variables $Q$ preserves reversibility and reflection symmetries, the vector field in \eqref{eq-dimbreak2} anti-commutes with the reverser $R$ and commutes with the reflection $S$. Consequently, the system \eqref{eq-dimbreak2} is reversible with reverser $R$ and the reflection symmetry $S$ implies that the subspace $\mathcal X_+^0$ given in \eqref{e:X0pm} is invariant. Taking $\mathcal X=\mathcal X^0_+$ and $\mathcal Y=\mathcal Y^1_\eps\cap \mathcal X^0_+$ the results in Theorem~\ref{lem-linop-spectral} imply that the hypotheses of Theorem~\ref{app-thm-Lyapunov} hold, for $\eps$ sufficiently small. This proves Theorem~\ref{thm-main-II}(\ref{thm-main-II-dimbreak}). 
	\end{proof}

	\section{Parameter Region~II}
	\label{sect-reg-III}
	
	The analysis done for $(\alpha,\beta)$ in Region~I can be easily transferred to the parameter Region~II. However, the final result is different because the linear dispersion relation \eqref{eq-lindisrel} possesses two positive roots for $(\alpha,\beta)$ in this parameter region. We point out the differences and then state the main result for this parameter region.
	
	Denote by $k_{*,1}$ and $k_{*,2}$ the two positive roots of the dispersion relation. Take $k_{*,1}<k_{*,2}$ and assume that $k_{*,2}/k_{*,1} \notin \ZZ$.
	
	First, the existence of two-dimensional periodic waves is proved in the same way, with the difference that we now find two geometrically distinct families of two-dimensional periodic waves $\{(\eta_{\eps, 1}(X), \phi_{\eps, 1}(X, Y))\}_{\eps \in (-\eps_0, \eps_0)}$ and $\{(\eta_{\eps, 2}(X), \phi_{\eps, 2}(X, Y))\}_{\eps \in (-\eps_0, \eps_0)}$ with wavenumbers $k_{\eps,1} = k_{*,1} + \O(\eps^2)$ and $k_{\eps,2}=k_{*,2} + \O(\eps^2)$, respectively. The expansions \eqref{eq-2Dsol-exp} remain valid with $k_*$ replaced by $k_{*,1}$ for the first family and by $k_{*,2}$ for the second family, and this is also the case for all other symbolic computations.
	
	Next, the analysis of the linear operator $L_\eps$ given in Section~\ref{s linear} stays the same until the last step of the proof of Theorem~\ref{lem-linop-spectral} which consists in showing that $m_{21}^{(2)}$ is negative. The formula for $m_{21}^{(2)}$ is the same, but the result is different for the first family of periodic waves. The analysis in Appendix~\ref{app-m212-sign} gives the conclusion that $m_{21}^{(2)}$ is negative for the second family of periodic waves, whereas for the first family of periodic waves it is negative only when  $2k_{*,1} > k_{*,2}$.  This condition is satisfied if and only if $(\alpha, \beta)$ belongs to the open region between $\Gamma_2$ and $\Gamma$ in Figure~\ref{fig-lindis}.
	
	Consequently, the two-dimensional periodic waves $(\eta_{\eps, 2}(X), \phi_{\eps, 2}(X,Y))$ are transversely linearly unstable, whereas the periodic waves $(\eta_{\eps, 1}(X), \phi_{\eps, 1}(X, Y))$ are unstable if $(\alpha, \beta)$ lies between $\Gamma_2$ and $\Gamma$. Notice that our approach does not allow us to conclude on stability because the general criterion in Theorem~\ref{thm-Godey} only provides sufficient conditions for instability. Finally, the dimension-breaking result holds for all linearly transversely waves. 
	
	We summarize these results in the following theorem.
	
	\begin{thm}[Region~II]\label{thm-main-III} Fix $(\alpha,\beta)$ in Region~II and let $k_{*,1}, k_{*,2}$ be the two positive roots of the dispersion relation~\eqref{eq-lindisrel}. Assume that $k_{*,1}< k_{*,2}$ and $k_{*,2}/k_{*,1} \notin \ZZ$. Denote by $\Gamma_2$ the $(\alpha,\beta)$-parameter curve for which $2k_{*,1} = k_{*,2}$.
		\begin{enumerate}[(i)]
			\item \label{thm-main-III-exist} (Existence) There exist $\eps_0 > 0$ and two geometrically distinct families of two-dimensional steady periodic waves 
			\[\begin{split}
			&\{(\eta_{\eps,1}(X), \phi_{\eps,1}(X, Y))\}_{\eps \in (-\eps_0, \eps_0)} \quad \text{and} \quad \{(\eta_{\eps,2}(X), \phi_{\eps, 2}(X, Y))\}_{\eps \in (-\eps_0, \eps_0)}
			\end{split}\]
			to the equations~\eqref{eq-hydro-nondim}--\eqref{eq-hydro-bc-nondim}, such that $(\eta_{0,i},\phi_{0,i}) =(0, 0)$ and $(\eta_{\eps, i}, \phi_{\eps, i})$ are periodic in $X$ with wavenumbers $k_{\eps, i} = k_{*,i}+ \O(\eps^2)$ for $i = 1,2$.
			\item \label{thm-main-III-unstable} (Transverse instability) There exists $\varepsilon_1>0$ such that for each $\eps\in(-\eps_1,\eps_1)$ the periodic solution $(\eta_{\eps, 2}, \phi_{\eps, 2})$ is transversely linearly unstable. The solution $(\eta_{\eps, 1}, \phi_{\eps, 1})$ is transversely linearly unstable if $\,2k_{*,1} > k_{*,2}$, which occurs for $(\alpha, \beta)$ in the open region between the curves $\Gamma_2$ and $\Gamma$.
			\item \label{thm-main-III-dimbreak} (Dimension-breaking bifurcation)  There exists $\varepsilon_2>0$ such that for each  transversely linearly unstable wave $(\eta_{\eps, i}, \phi_{\eps,i})$ with $\eps\in(-\eps_2,\eps_2)$, $i=1,2$, there exist $\delta_\eps>0$, $\ell_{\eps,i}^*>0$, and a family of three-dimensional doubly periodic waves $\{(\eta_{\eps,i}^\delta(X, z),$ $\phi_{\eps,i}^\delta(X, Y, z))\}_{\delta \in (-\delta_{\eps,i}, \delta_{\eps,i})}$, with wavenumber $k_{*,i}$ in $X$ and wavenumber $\ell_\delta=\ell_{\eps,i}^*+\O(\delta^2)$ in $z$, bifurcating from the periodic solution  $(\eta_{\eps, i}, \phi_{\eps,i})$.
		\end{enumerate}
	\end{thm}

	\subsection*{Acknowledgement}
	M.H. gratefully acknowledges support and generous hospitality by Lund University through the Hedda Andersson visiting chair.
	M.H. was partially supported by the project Optimal (ANR-20-CE30-0004) and the EUR EIPHI program (ANR-17-EURE-0002).
	T.T. and E.W. gratefully acknowledge the support by the Swedish Research Council, grant no.~2016--04999.

	\appendix
	
	\section{Lyapunov center theorem}
	\label{app-Lyapunov}
	
	We state a non-resonant version of the Lyapunov center theorem for reversible systems which is a particular case of the more general version from~\cite{Bagri}.
	\begin{thm}
		\label{app-thm-Lyapunov} Let $\mathcal X$ and $\mathcal Y$ be real Banach spaces such that $\mathcal Y$ is continuously embedded in $\mathcal X$. Consider the evolutionary equation
		\begin{equation}\label{eq-app-abstract-evo}
		\frac{\dif U}{\dif t} = F(U),
		\end{equation}
		where $F  \in \mathscr C^6(\mathcal U, \mathcal X)$  with $\mathcal U \subset \mathcal Y$ a neighborhood of $\,0$. Assume that $F(0)=0$ and that the following properties hold:
		\begin{enumerate}[(i)]
			\item there exists an involution $R\in\mathcal L(\mathcal X)\cap\mathcal L(\mathcal Y)$ which anticommutes with $F$, i.e., the equation \eqref{eq-app-abstract-evo} is reversible;
			\item the linear operator $L \coloneqq {\dif}F[0]$ possesses a pair of simple eigenvalues  $\pm{\ii}\omega_0$ with $\omega_0 > 0$;
			\item for each $n \in \ZZ \setminus \{-1, 1\}$,  ${\ii}n\omega_0$ belongs to the resolvent set of $L$;
			\item there exists a positive constant $C$ such that
			\begin{equation}\label{e:resolvent}
			\|(L-{\ii}n\omega_0 \mathbb I )^{-1}\|_{\mathcal L(\mathcal X, \mathcal X)} \leq \frac{C}{|n|},\quad
			\|(L-{\ii}n \omega_0 \mathbb I)^{-1}\|_{\mathcal L(\mathcal X, \mathcal Y)} \leq C,
			\end{equation}
			as $n \to \infty$.
		\end{enumerate}
		Then, there exists a neighborhood $E\subset \R$ of $\, 0$ and a $\mathscr C^4$-curve $\{U(\eps),\omega(\eps)\}_{\eps \in E}$ where $U(\eps)$ is a real periodic solution to~\eqref{eq-app-abstract-evo} with period $2\pi/\omega(\eps)$. Furthermore, $(U(0),\omega(0))=(0,\omega_0)$.
	\end{thm}
	
	In the version of the above theorem from~\cite{Bagri} the curve $\{U(\eps),\omega(\eps)\}_{\eps \in E}$ was only continuously differentiable and the vector field $F$ was of class $\mathscr C^3$. For our purposes we need at least a $\mathscr C^4$-dependence on $\eps$ and we therefore assume that $F$ is of class $\mathscr C^6$.
	
	\section{Expansion of the two-dimensional periodic waves}
	\label{app-2D-calc}
	
	For this computation it is more convenient to use the original system~\eqref{eq-hydro-nondim}--\eqref{eq-hydro-bc-nondim} instead of the dynamical system~\eqref{eq-hydro-evo}. Restricting to two-dimensional steady solutions, we make the change of variables
	\[
	X=kx,\quad Y=y(1+\eta(X)),\quad \eta(X)=\widetilde\eta(x),\quad\phi(X,Y) = \widetilde\Phi(x, y). 
	\]
	Dropping the tildes we obtain the equations
	\begin{equation} \label{eq-Laplaceflat} \begin{split}
	k^2 \Phi_{xx}& + \frac{1}{(1+\eta)^2}\Phi_{yy} - 2k^2 \frac{y \eta_x}{1+\eta} \Phi_{xy} \\
	& + k^2 \left(\frac{y \eta_x}{1+\eta}\right)^2 \Phi_{yy} + k^2 \left(\frac{2y\eta_x^2}{(1+\eta)^2}-\frac{y\eta_{xx}}{1+\eta}\right) \Phi_y = 0,
	\end{split}
	\end{equation}
	for $0 < y < 1$ with boundary conditions
	\begin{equation} \begin{aligned}
	\label{eq-bcflat}
	\Phi_y &= 0& &\text{on $y=0$},&\\
	\Phi_y &= (1+\eta)(-k \eta_x+ k^2\eta_x \Phi_x)- k ^2 \eta_x^2 \Phi_y& &\text{on $y=1$,}&\\
	\alpha \eta& - k  \left(\Phi_x - \frac{\eta_x}{1+\eta}\Phi_y \right) - \beta k^2 \left(\frac{\eta_x}{(1+k^2\eta_x^2)^{1/2}}\right)_x& && \\
	& + \frac{1}{2}\left(k^2\left(\Phi_x - \frac{\eta_x}{1+\eta} \Phi_y\right)^2 + \frac{\Phi_y^2}{(1+\eta)^2}\right) = 0& \quad &\text{on $y=1$}.&
	\end{aligned}
	\end{equation}
	The scaled periodic wave $(\widetilde\eta_\eps(x),\widetilde\Phi_\eps(x,y))$ satisfies these equations for $k=k_\eps$. We insert the expansions~\eqref{eq-2Dsol-exp} into equations~\eqref{eq-Laplaceflat}--\eqref{eq-bcflat} and expand the resulting equations in $\eps$. We restrict to solutions with $\widetilde\eta_\eps$ even in $x$ and $\widetilde\Phi_\eps$ odd in $x$.
	
	At order $\O(\eps)$ we find the following equations for $\eta_1$ and $\Phi_1$:
	\begin{equation}\label{eq-order-eps}
	\begin{aligned}
	&k_*^2 \Phi_{1xx} + \Phi_{1yy} = 0&\quad  &\text{for $0 < y < 1$},&\\
	&\Phi_{1y}= 0& \quad &\text{on $y=0$},&\\
	&\Phi_{1y}=- k_*\eta_{1x} &\quad  &\text{on $y=1$}, & \\
	&\alpha \eta_1 - k_* \Phi_{1x}|_{y=1} - \beta k_*^2 \eta_{1xx} = 0&\quad  & \text{on $y=1$}.&
	\end{aligned}
	\end{equation}
	Taking Fourier series in $x$,
	\[
	\eta_1(x) = \sum_{n=0}^\infty \eta_{1n} \cos(nx), \quad \Phi_1(x,y) = \sum_{n=1}^\infty \phi_{1n}(y)\sin(nx),
	\]
	we obtain the solvability condition $\mathcal D(nk_*)=0$ where $\mathcal D$ is the linear dispersion relation in~\eqref{eq-lindisrel}. Consequently, $n=0$ which gives solutions which are constant in $x$, only, and $n=\pm 1$ which gives the formulas for $\eta_1$ and $\Phi_1$ in \eqref{e:expan}.
	
	Next, at order $\O(\eps^2)$ we obtain the equations for $\eta_2$ and $\Phi_2$,
	\begin{equation}
	\begin{aligned}
	&k_*^2 \Phi_{2xx} + \Phi_{2yy} =2\eta_1 \Phi_{1yy} + 2k_*^2 y \eta_{1x} \Phi_{1xy} + k^2_* y \eta_{1xx} \Phi_{1y}& \quad &\text{for $0 < y < 1$},& \\
	&\Phi_{2y} =0 &\quad  &\text{on $y=0$},&\\
	&\Phi_{2y}+k_*\eta_{2x}=  k^2_* \eta_{1x} \Phi_{1x}-k_*\eta_1 \eta_{1x}& &\text{on $y=1$},&\\
	&\alpha \eta_2 - \beta k_*^2 \eta_{2xx} - k_* \Phi_{2x} = {-}k_* \eta_{1x} \Phi_{1y} - \tfrac{1}{2} (k_*^2 \Phi_{1x}^2 + \Phi_{1y}^2)& \quad &\text{on $y=1$}.&
	\end{aligned}
	\end{equation}
	Inserting the explicit formulas for $\eta_1$ and $\Phi_1$ we obtain 
	\begin{equation}
	\begin{aligned}
	&k_*^2 \Phi_{2xx} + \Phi_{2yy} =k_*^2 \sinh(k_*)\sin(2x)\left(\cosh(k_*y)-\tfrac{3}{2}k_* y \sinh(k_*y)\right)&  &\text{for $0 < y < 1$},& \\
	&\Phi_{2y} =0 & &\text{on $y=0$},&\\
	&\Phi_{2y}+k_*\eta_{2x} = \left(\tfrac{k_*}{2}\sinh^2(k_*)-\tfrac{k_*^2}{2} \sinh(k_*)\cosh(k_*)\right) \sin(2x)& &\text{on $y=1$},&\\
	&\alpha \eta_2 - \beta k_*^2 \eta_{2xx} - k_* \Phi_{2x} = {-}\tfrac{k_*^2}{4}\cosh(2k_*)\cos(2x)-\tfrac{k_*^2}{4}& &\text{on $y=1$}.&
	\end{aligned}
	\end{equation}
	Observing that the right-hand side only involves the second Fourier mode, we find the formulas for $\eta_2$ and $\Phi_2$ in \eqref{e:expan}.  
	
	Finally, the coefficient $k_2$ is determined from the expansion at order $\mathcal O(\eps^3)$,
	\begin{equation*}
	\begin{aligned}
	&k_*^2 \Phi_{3xx} + \Phi_{3yy} =-2k_*k_2\Phi_{1xx}-3\eta_1^2 \Phi_{1yy} + 2\eta_2 \Phi_{1yy} + 2\eta_1 \Phi_{2yy} &\quad && \\
	& \phantom{k_*^2 \Phi_{3xx} + \Phi_{3yy} = }
	+2k_*^2y(\eta_{2x}\Phi_{1xy}+\eta_{1x}\Phi_{2xy}-\eta_{1x}\eta_1\Phi_{1xy}-\eta_{1x}^2\Phi_{1y}) & && \\
	&  \phantom{k_*^2 \Phi_{3xx} + \Phi_{3yy} = }
	-k_*^2y^2\eta_{1x}^2 \Phi_{1yy} + yk_*^2(\eta_{1xx}\Phi_{2y}+\eta_{2xx}\Phi_{1y}-\eta_{1xx}\eta_1\Phi_{1y})& &\\
	\end{aligned}
	\end{equation*}
	for $0 < y < 1$, and
	\begin{equation*}
	\begin{aligned}
	&\Phi_{3y} =0 & &\text{on $y=0$},&\\
	&\Phi_{3y}+k_*\eta_{3x} = -k_*(\eta_1\eta_{2x}+\eta_{1x}\eta_2) - k_2\eta_{1x}& &&\\
	& \phantom{\Phi_{3y}+k_*\eta_{3x} = } + k_*^2(\eta_{1x}\Phi_{2x}+\eta_{2x}\Phi_{1x}+\eta_1\eta_{1x}\Phi_{1x})-k_*^2\eta_{1x}^2\Phi_{1y} & &\text{on $y=1$}, & \\
	&\alpha \eta_3 - \beta k_*^2 \eta_{3xx} - k_* \Phi_{3x} =k_2\Phi_{1x}-k_*(\eta_{1x}\Phi_{2y}+\eta_{2x}\Phi_{1y}-\eta_{1x}\eta_1 \Phi_{1y})&  &&\\
	& \phantom{\alpha \eta_3 - \beta k_*^2 \eta_{3xx} - k_* \Phi_{3x} = }
	+2\beta k_*k_2 \eta_{1xx} - \frac{3}{2}\beta k_*^4 \eta_{1xx} \eta^2_{1x}& &&\\
	&\phantom{\alpha \eta_3 - \beta k_*^2 \eta_{3xx} - k_* \Phi_{3x} = }
	-k_*^2(\Phi_{1x}\Phi_{2x}-\eta_{1x}\Phi_{1x}\Phi_{1y})-\Phi_{1y}\Phi_{2y}+\eta_1\Phi_{1y}^2& &\text{on $y=1$}.& \\
	\end{aligned}
	\end{equation*}
	The right hand sides involving only Fourier modes $1$ and $3$,  we  write
	\[\eta_3(x)=\eta_{31}\cos(x)+ \eta_{33}\cos(3x), \quad \Phi_{3}(x,y)=\phi_{31}(y)\sin(x)+\phi_{33}(y)\sin(3x),\]
	and the resulting systems for $(\eta_{31},\phi_{31})$ and  $(\eta_{33},\phi_{33})$ are decoupled. The coefficient $k_2$ only appears in the terms with Fourier mode $1$ so that it is enough to solve the equations for  $(\eta_{31},\phi_{31})$. These equations are of the form
	\begin{equation}\label{e:O3}
	\begin{aligned}
	-k_*^2\phi_{31}(y) + \phi_{31}''(y) &=F_3& \quad &\text{for $0 < y < 1$},&\\
	(\alpha + \beta k_*^2)\eta_{31}-k_*\phi_{31}(1) &=g_3,& & & \\
	\phi_{31}'(0) &=0,& && \\
	\phi_{31}'(1)-k_*\eta_{31} &=f_3,& &&
	\end{aligned}
	\end{equation}
	where, after computations, we find the explicit formulas
	\[
	\begin{split}
	F_3(y) = & \left(\frac{k_*^2\sinh^2(k_*)}{4}-\frac{k_*^4}{2\alpha}-\frac{k_*^3}4\left(c(k_*)+1\right)\sinh(2k_*)+2k_*k_2\right)\cosh(k_*y) \\
	& +\frac{k_*^3\sinh^2(k_*)}{2}y\sinh(k_*y) + k_*^3 c(k_*)\sinh(k_*)\cosh(2k_*y) \\
	&+ \frac{3k_*^4c(k_*)\sinh(k_*)}4y \sinh(2k_*y),
	\end{split}\]
	and
	\[\begin{split}
	g_3 = & \left(\frac{k_*^3}{16}-\frac{k_*^3c(k_*)}4 +k_2\right)\cosh(k_*)+ \left(-\frac{9}{32}\beta k_*^4 + 2\beta k_* k_2\right) \sinh(k_*) \\
	&-\frac{k_*^3}{16}\cosh(3k_*)+\frac{3}{32}\beta k_*^4 \sinh(3k_*),\\[0.5ex]
	f_3 = & -\frac{k_*^2c(k_*)}{16}\cosh(k_*) + \left(-\frac{k_*^3c(k_*)}4+\frac{k_*^3}{16}-\frac{k_*^3}{4\alpha}+k_2\right)\sinh(k_*)\\
	& + \frac{k_*^2c(k_*)}{16}\cosh(3k_*) +\frac{3}{16}k_*^3 \sinh(3k_*).
	\end{split}\]
	
	Observe that the system~\eqref{e:O3} is equivalent to the linear nonhomogeneous equation 
	\begin{equation}\label{e:linO3}
	L_0\begin{pmatrix} \eta_{31}\cos(x) \\ 0 \\ \phi_{31}(y)\sin(x)\\ 0 \end{pmatrix} =\begin{pmatrix} 0 \\ g_3 \cos(x) \\ 0 \\ -F_3(y)\sin(x)\end{pmatrix}.
	\end{equation}
	where $L_0$ is the linear operator from Section~\ref{s linear}, together with the linear nonhomogeneous boundary conditions 
	\begin{equation}\label{e:linO3bc}
	\phi_{31}'(0) = 0 \quad\text{and}\quad \phi_{31}'(1)-k_*\eta_{31} = f_3.
	\end{equation}
	Consider the dual vector
	\begin{equation}\label{e:zeta*}
	\zeta_+^*=\begin{pmatrix} 0 \\ \sinh(k_*)\cos(x) \\ 0 \\ \cosh(k_*y)\sin(x)\end{pmatrix},
	\end{equation}
	which belongs to the kernel of the adjoint operator $L_0^*$. Taking into account the nonhomo- geneous boundary conditions \eqref{e:linO3bc}, a direct computation of the scalar product of \eqref{e:linO3} with  $\zeta_+^*$ leads to the solvability condition
	\begin{equation}\label{e:scO3}
	f_3 \cosh(k_*) + g_3\sinh(k_*) = \int_0^1 F_3(y)\cosh(k_*y)\dif y.
	\end{equation}
	Indeed, integrating twice by parts we find 
	\[\begin{split}
	&\int_0^1 F_3(y)\cosh(k_*y)\dif y - g_3 \sinh(k_*)\\
	&\quad = \int_0^1 \left(k_*^2\phi_{31}(y)-\phi_{31}''(y)\right) \cosh(k_*y) \dif y - g_3\sinh(k_*)\\
	&\quad = \phi_{31}'(1)\cosh(k_*)-k_*\phi_{31}(1)\sinh(k_*) - g_3\sinh(k_*)\\
	&\quad = (f_3+k_*\eta_{31})\cosh(k_*) - (\alpha + \beta k_*^2)\eta_{31}\sinh(k_*) \\
	&\quad = f_3 \cosh(k_*),
	\end{split}\]
	where we have also used  the linear dispersion relation $\mathcal D(k_*)=0$. 
	Replacing the explicit formulas for $F_3$, $f_3$, and $g_3$ into the solvability condition \eqref{e:scO3} and solving for $k_2$  we obtain the formula~\eqref{e:k2}.

	\section{Computation of the coefficient $\boldsymbol{m_{21}^{(2)}}$}
	\label{app-m212}
	
	We prove the equality \eqref{eq-m212} which connects the coefficient ${m_{21}^{(2)}}$ with the coefficient $k_2$ in the expansion of the wavenumber $k_\eps$ of the periodic wave.
	
	We emphasize the dependence on $k$ of the vector field $F$ in \eqref{eq-hydro-evo} by writing $F(U,k)$ and similarly for $B$ in the boundary conditions \eqref{eq-hydro-bc-evo} we write $B(U,k)$. Setting
	\[
	\widetilde B(U,k)=B(U,k) - \Phi_{y},
	\]
	the two-dimensional periodic wave $U_\eps$ given in \eqref{e:Ueps0} satisfies
	\begin{equation}
	\begin{cases}\label{eq-hydro-syst}F(U_\eps, k_\eps) = 0, \quad y\in (0,1) \\
	\widetilde{B}(U_\eps,k_\eps) = 0, \quad y = 0,1.
	\end{cases}
	\end{equation}
	The linear operator $L_\eps$ is equal to ${\D}_U F[U_\eps, k_\eps]$ with boundary conditions ${\D}_U \widetilde{B}[U_\eps, k_\eps]= 0.$ To determine $m_{21}^{(2)}$, we study the problem
	\begin{equation}\label{eq-spectral}L_\eps \zeta_+(\eps) = m_{11}(\eps)\zeta_+(\eps)+m_{21}(\eps)\psi_+(\eps),\end{equation}
	with the boundary conditions
	\begin{equation}\label{eq-spectral-bc}{\D}_U\widetilde{B}[U_\eps, k_\eps]\zeta_+(\eps) = 0.\end{equation}
	We will make a connection between~\eqref{eq-hydro-syst} and~\eqref{eq-spectral}--\eqref{eq-spectral-bc} using the expansions in $\eps$ of $F, \widetilde B,$ $U_\eps$, $k_\eps$, $\zeta_+(\eps)$, $\psi_+(\eps)$, $m_{11}(\eps)$ and $m_{21}(\eps)$.
	
	For the system~\eqref{eq-hydro-syst}, we write
	\[
	\begin{split}
	&U_\eps = \eps U_1 + \eps^2 U_2 + \eps^3U_3 + \mathcal O(\eps^4),\\
	&k_\eps = k_* + \eps^2 k_2 + \mathcal O(\eps^4),\\
	\end{split}
	\]
	and  take the Taylor expansion of $F(U,k)$ at $(0,k_*)$,
	\[
	\begin{split}
	F(U, k) &= {\D}_U F[0, k_*] U + \frac{1}{2}{\D}_{UU}^2F[0,k_*](U,U) \\
	& \quad + {\D}_{Uk}^2F[0,k_*](U, k-k_*) + \frac{1}{6}{\D}_{UUU}^3F[0,k_*](U,U,U) \\
	& \quad +  \mathcal O\left(|k-k_*|^2\|U\|_{\mathcal X^1}+|k-k_*|\|U\|_{\mathcal X^1}^2 + \|U\|_{\mathcal X^1}^4\right),
	\end{split}
	\]
	where all derivatives ${\D}^{(q)}_kF[0,k_*] = 0$ \color{black} because $F(0,k)=0$ for all $k$. A similar Taylor expansion can be written for the nonlinear boundary condition ${\widetilde{B}}$. In particular $L_0 = {\D}_UF[0,k_*]$ with boundary conditions $\widetilde{B}_0 U = {\D}_U\widetilde{B}[0,k_*]U = 0$ on $y=0,1$.
	Inserting these expansions into \eqref{eq-hydro-syst}, at order $\O(\eps)$, we find
	\begin{equation}\label{eq-hydro-order-eps}\begin{cases}
	L_0U_1 =0, \quad y\in (0,1),\\
	\widetilde{B}_0U_1 =0, \quad y=0,1.
	\end{cases}\end{equation}
	Thus $U_1$ belongs to the kernel of $L_0$ and we choose $U_1=\zeta_+$ which is in agreement with the expansion of $U_\eps$ given by \eqref{eq-2Dsol-exp}. At order $\O(\eps^2)$, we obtain the system
	\begin{equation}\label{eq-hydro-order-eps2}\begin{cases}L_0 U_2 = - \frac{1}{2}{\D}_{UU}^2F[0,k_*](U_1,U_1), \quad y\in (0,1),\\
	\widetilde{B}_0U_2 = -\frac{1}{2}{\D}_{UU}^2\widetilde{B}[0,k_*](U_1,U_1), \quad y = 0,1.
	\end{cases}\end{equation}
	and at order $\O(\eps^3)$ we find
	\begin{equation}\label{eq-hydro-order-eps3}
	\begin{cases}
	L_0U_3 = -{\D}_{UU}^2F[0,k_*](U_1,U_2)-{\D}_{Uk}^2F[0,k_*](U_1,k_2) \\
	\phantom{L_0U_3 =} - \frac{1}{6}{\D}_{UUU}^3F[0,k_*](U_1,U_1,U_1), \qquad \qquad y\in(0,1),\\
	\widetilde{B}_0U_3 = -{\D}_{UU}^2\widetilde{B}[0,k_*](U_1,U_2)-{\D}_{Uk}^2\widetilde{B}[0,k_*](U_1,k_2) \\
	\phantom{\widetilde{B}_0U_3 =} - \frac{1}{6}{\D}_{UUU}^3\widetilde{B}[0,k_*](U_1,U_1,U_1), \qquad \qquad y=0,1.
	\end{cases}
	\end{equation}
	The nonhomogeneous systems~\eqref{eq-hydro-order-eps2} and~\eqref{eq-hydro-order-eps3} have each a unique solution $ U_2$ and $U_3$, respectively, up to an element in the kernel of $L_0$. Furthermore, $U_2=(\eta_2,0,\Phi_2,0)^{\text{T}}$ in agreement with the expansion of $U_\eps$ given by \eqref{eq-2Dsol-exp}. 
	
	Similarly, for the spectral problem~\eqref{eq-spectral}--\eqref{eq-spectral-bc} we take
	\[
	\begin{split}
	&\zeta_+(\eps)\, = \zeta_+ \, \,+ \eps \zeta_{+,1} \, + \eps^2 \zeta_{+,2} \, + \O(\eps^3),\\
	&\psi_+(\eps) = \psi_+ + \eps\psi_{+,1} + \eps^2 \psi_{+,2} + \O(\eps^3),\\
	&m_{11}(\eps) = \eps^2 m_{11}^{(2)} + \O(\eps^4),\\
	&m_{21}(\eps) = \eps^2 m_{21}^{(2)} + \O(\eps^4),
	\end{split}
	\]
	the Taylor expansion
	\[
	\begin{split}
	L_\eps &= {\D}_UF[U_\eps, k_\eps](\, \cdot\,) \\
	&= {\D}_U F[0,k_*](\, \cdot \,) + {\D}_{UU}^2F[0,k_*](U_\eps, \, \cdot \,) + {\D}_{Uk}^2F[0,k_*](\, \cdot \,, k_\eps-k_*) \\
	& \quad + \frac{1}{2}{\D}_{UUU}^3F[0,k_*](U_\eps, U_\eps, \, \cdot \,) +  \mathcal O\Big(|k_\eps-k_*|^2+|k_\eps-k_*|\|U_\eps\|_{\mathcal X^1} + \|U_\eps\|_{\mathcal X^1}^3\Big),
	\end{split}
	\]
	and a similar Taylor expansion can be written for ${\D}_U\widetilde{B}[U_\eps,k_\eps]$. Inserting these expansions in equations~\eqref{eq-spectral}--\eqref{eq-spectral-bc}, we find at order $\O(1)$ the system 
	\[\begin{cases}
	L_0\zeta_+ = 0, \quad y \in (0,1),\\
	\widetilde{B}_0\zeta_+ = 0, \quad y = 0,1.
	\end{cases}\]
	This is precisely the system~\eqref{eq-hydro-order-eps} and we may choose $U_1=\zeta_+$. Next, at order $\O(\eps)$, we find
	\begin{equation*}
	\begin{cases}
	L_0 \zeta_{+,1} = -{\D}_{UU}^2F[0,k_*](U_1, U_1), \quad y \in (0,1),\\
	\widetilde{B}_0 \zeta_{+,1} = -{\D}_{UU}^2\widetilde{B}[0,k_*](U_1, U_1), \quad y = 0,1.
	\end{cases}
	\end{equation*}
	Comparing to~\eqref{eq-hydro-order-eps2}, we choose $\zeta_{+,1} = 2U_2$. Finally, at order $\O(\eps^2)$, we have the system
	\begin{equation*}
	\begin{cases}
	L_0 \zeta_{+,2} = -3{\D}_{UU}^2F[0,k_*](U_1, U_2) - {\D}_{Uk}^2F[0,k_*](U_1, k_2)\\
	\phantom{L_0 \zeta_{+,2} = }- \frac{1}{2}{\D}_{UUU}^3F[0,k_*](U_1,U_1,U_1) + m_{11}^{(2)}\zeta_+ + m_{21}^{(2)}\psi_+, \quad y \in (0,1),\\
	\widetilde{B}_0 \zeta_{+,2} = -3{\D}_{UU}^2\widetilde{B}[0,k_*](U_1, U_2) - {\D}_{Uk}^2\widetilde{B}[0,k_*](U_1, k_2)\\
	\phantom{\widetilde{B}_0 \zeta_{+,2} = }- \frac{1}{2}{\D}_{UUU}^3\widetilde{B}[0,k_*](U_1,U_1,U_1), \hfill  y = 0,1.\phantom{()}
	\end{cases}
	\end{equation*}
	Substracting the system~\eqref{eq-hydro-order-eps3} from the one above we obtain the system
	\begin{equation}\label{e:k2m21}
	\begin{cases}
	L_0 (\zeta_{+,2}-3U_3) = 2 {\D}_{Uk}^2F[0,k_*](U_1, k_2)+ m_{11}^{(2)}\zeta_+ + m_{21}^{(2)}\psi_+, \quad y \in (0,1),\\
	\widetilde{B}_0 (\zeta_{+,2}-3U_3) = 2 {\D}_{Uk}^2\widetilde{B}[0,k_*](U_1, k_2), \hspace{3.68cm} y = 0,1,\phantom{()}
	\end{cases}
	\end{equation}
	in which $k_2$ and $ m_{21}^{(2)}$ appear in the right hand sides of the two equalities.
	
	We obtain the connection between  $k_2$ and $ m_{21}^{(2)}$ by taking the scalar product of the first equation in \eqref{e:k2m21} with the dual vector $\zeta_+^*$ given by \eqref{e:zeta*} which belongs to the kernel of the adjoint operator $L_0^*$. This leads to the equality 
	\begin{equation}\label{eq-m212-pre}
	m_{21}^{(2)} = \frac{\langle L_0(\zeta_{+,2}-3U_3),\zeta_+^*\rangle}{\langle \psi_+, \zeta_+^*\rangle}-2\frac{\langle {\D}_{Uk}^2F[0,k_*](U_1,k_2), \zeta_+^*\rangle}{\langle \psi_+,\zeta_+^*\rangle}.
	\end{equation}
	It remains to explicitly compute the right hand side of this equality.
	
	First, we compute the denominator
	\[\begin{split}\langle \psi_+, \zeta_+^* \rangle &= \int_0^{2\pi} \int_0^1 \cosh^2(k_*y)\sin^2(x)\dif y \dif x + \int_0^{2\pi} \beta \sinh^2(k_*)\cos^2(x)  \dif x \\
	&= \pi \left(\frac{1}{2} + \frac{\cosh(k_*)\sinh(k_*)}{2k_*} + \beta \sinh^2(k_*)\right).
	\end{split}\]
	Next, observe that $\zeta_{+,2}=(\eta_{+,2},0, \Phi_{+,2},0)^{\TT}$, and similarly the second and fourth components of $U_3$ vanish. Setting
	\[
	\zeta_{+,2}-3U_3 = (\eta, 0, \Phi, 0)^{\TT},
	\]
	and taking into account the boundary conditions from  \eqref{e:k2m21}, we find
	\[\begin{split}
	\langle L_0(\zeta_{+,2}-3U_3), \zeta_+^* \rangle 
	&=  \int_0^1 \int_0^{2\pi} \left(-k_*^2\Phi_{xx}-\Phi_{yy}\right)\cosh(k_*y)\sin(x) \dif x \dif y \\
	& \quad  + \int_0^{2\pi} \left(\alpha \eta -\beta k_*^2 \eta_{xx} - k_*\Phi_x\big|_{y=1}\right) \sinh(k_*)\cos(x) \dif x\\
	&= -2\int_0^{2\pi} k_2 \eta_{1x} \cosh(k_*)\sin(x)\dif x \\
	&= 2k_2\int_0^{2\pi} \sinh(k_*)\cosh(k_*)\sin^2(x)\dif x\\
	&=\pi k_2 \sinh(2k_*).
	\end{split}\]
	where we have integrated twice by parts and used the linear dispersion relation $\mathcal D(k_*)=0$. 
	Finally,
	\[{\D}_{Uk}^2F[0,k_*](U_1, k_2) = k_2 \begin{pmatrix} 0 \\ -2\beta k_* \eta_{1xx} - \Phi_{1x}|_{y=1} \\ 0 \\ -2k_*\Phi_{1xx}\end{pmatrix},\]
	which gives 
	\[\langle {\D}_{Uk}^2F[0,k_*](U_1, k_2), \zeta_+^*\rangle = \pi k_2 \left(2\beta k_* \sinh^2(k_*)+ k_*\right).\]
	Replacing these explicit formulas into \eqref{eq-m212-pre} gives the equality~\eqref{eq-m212}.
	
	\section{Sign of the coefficient $\boldsymbol{m_{21}^{(2)}}$}
	\label{app-m212-sign}

	Consider $\widetilde{m}_{21}^{(2)}$ given by~\eqref{eq-coeff-m212} and assume $k_*>0$ is such that $\mathcal D(k_*)=0$. Let
	\[\sigma = \tanh(k_*) \quad \text{and} \quad \widetilde{T} = \frac{\beta}{\alpha} k_*^2.\]
	Note that $\sigma \in (0, 1)$ and $\widetilde{T} > 0$. A straightforward but lengthy calculation shows that 
	\begin{equation}
	\label{eq-m212-formula}
	\begin{split}\widetilde{m}_{21}^{(2)}&=\frac{8k_*}{(1+\widetilde{T})(1-\sigma^2)^2} \Bigg\{\frac{(1-\sigma^2)(9-\sigma^2)+\widetilde{T}(3-\sigma^2)(7-\sigma^2)}{\sigma^2-\widetilde{T}(3-\sigma^2)}\\
	& \hspace{5cm} \, \, \, +8\sigma^2 - 2(1-\sigma^2)^2(1+\widetilde{T})-\frac{3\sigma^2 \widetilde{T}}{1+\widetilde{T}}\Bigg\}\\
	&= \frac{8k_*}{(1+\widetilde{T})^2 (1-\sigma^2)^2} \cdot \frac{1}{\sigma^2-\widetilde{T}(3-\sigma^2)} \cdot \left(a_3 \widetilde{T}^3 + a_2\widetilde{T}^2 + a_1\widetilde{T} + a_0\right),\end{split}
	\end{equation}
	in which 
	\[\begin{aligned}
	a_3 &= -2\sigma^6+10\sigma^4-14\sigma^2+6, \quad &a_2 &=-6\sigma^6+30\sigma^4-55\sigma^2+33,\\
	a_1 &=-6\sigma^6+33\sigma^4-62\sigma^2+36, \quad &a_0 &= -2\sigma^6+13\sigma^4-12\sigma^2+9.
	\end{aligned}\]
	The coefficients $a_i$ are all positive for $\sigma \in (0, 1)$ and $i=0,1,2,3$. It is now clear that all factors in the above formula for $\widetilde{m}_{21}^{(2)}$ are positive, except for $(\sigma^2-\widetilde{T}(3-\sigma^2))^{-1}$, and the sign of $\widetilde{m}_{21}^{(2)}$ is precisely the sign of $\sigma^2 - \widetilde{T}(3-\sigma^2).$ We have
	\[\text{sign}(\sigma^2-\widetilde{T}(3-\sigma^2))=-\text{sign}(\mathcal D(2k_*)).\]
	In Region I, the linear dispersion relation $\mathcal D(k)=0$ has a unique positive simple root $k=k_*$. It follows that the smooth function $k\mapsto \mathcal D(2k)$ changes signs exactly once at $k=k_*/2$. Since $\mathcal D(0) =0$ and $\mathcal D'(0) < 0$, we conclude that $\mathcal D(2k)$ is negative for $k \in (0, k_*/2)$, and positive for $k \in (k_*/2, \infty)$. In particular, we have $\mathcal D(2k_*) > 0$. In Region II, the linear dispersion relation $\mathcal D(k)=0$ has two positive simple roots $k=k_{*,1}$ and $k=k_{*,2}$, $0 < k_{*,1} < k_{*,2}$. This means that the smooth function $\mathcal D(2k)$ changes signs exactly twice, first at $k=k_{*,1}/2$ and then at $k=k_{*,2}/2$. Since $\mathcal D(0) = 0$ and $\mathcal D'(0)>0$, we conclude that $\mathcal D(2k)$ is positive on $(0, k_{*,1}/2)$, negative on $(k_{*,1}/2, k_{*,2}/2)$ and positive on $(k_{*,2}/2, \infty)$. This implies that $\mathcal D(2k_{*,2})>0$, whereas 
	\[\begin{split}
	&\mathcal D(2k_{*,1}) > 0 \quad \text{if} \quad k_{*,1}>k_{*,2}/2,\\
	&\mathcal D(2k_{*,1})<0 \quad \text{if} \quad k_{*,1}<k_{*,2}/2.
	\end{split}\]
	The parameter region in which we have $\mathcal D(2k_{*,1}) > 0$ is precisely the open region between the curve $\Gamma$ and $\Gamma_2$. We summarize our findings below.
	\begin{prop} 
		\label{prop-coeff-sign}
		If $(\alpha,\beta)$ belongs to Region I, the coefficient $m_{21}^{(2)}$ is negative. If $(\alpha,\beta)$ belongs to Region II, then
		\begin{itemize}
			\item $m_{21}^{(2)}$ is negative for both wavenumbers $k_{*,1}$ and $k_{*,2}$ if $(\alpha, \beta)$ lies in the open region between $\Gamma$ and $\Gamma_2$,
			\item $m_{21}^{(2)}$ is positive for $k_{*,1}$ and negative for $k_{*,2}$ if $(\alpha, \beta)$ lies to the left of $\,\Gamma_2$.
		\end{itemize}  
	\end{prop}
	
	\section{Formal derivation of instability criterion}	
	\label{app-formal-derivation}
	
	We demonstrate how our instability criterion can formally be derived based on the Davey-Stewartson model. This is most easily expressed in the original physical frame of reference and dimensional variables\footnote{In this appendix, the wavenumbers $k$ and $k_*$ are dimensional and thus $kh$ and $k_*h$ correspond to the wavenumbers used elsewhere in this paper.}.
	Making a multiple-scales Ansatz
	\begin{align*}
	\phi&=\eps \frac{\sqrt{gk}}{k^2} A(\eps kz, \eps^2 \sqrt{gk} t)\frac{\cosh(k(y+h))}{\cosh(kh)} e^{\ii (k x-\omega t)} + \text{c.c.} + \cdots,\\
	\eta&= \eps \frac{\ii\omega \sqrt{gk}}{k^2(g+k^2T/\rho)} A(\eps kz, \eps^2 \sqrt{gk} t)e^{\ii (kx-\omega t)} + \text{c.c.} + \cdots,
	\end{align*}
	one obtains at cubic order in $\eps$ an evolution equation of NLS type in the stretched transverse direction $\zeta=\eps kz$, 
	\[
	\ii\frac{\partial A}{\partial \tau}+\mu \frac{\partial^2 A}{\partial \zeta^2}=\chi |A|^2 A.
	\]
	In this approximation, the families of two-dimensional periodic waves which we study can be represented by the $\zeta$-independent solution $A=\exp(-\ii  \chi \tau)$, and a classical result says that it is unstable precisely when  $\mu \chi<0$. 
	According to Ablowitz \& Segur \cite{AblowitzSegur79} (see also Djordjevic \& Redekopp \cite{DR}) $\mu$ is always positive, while $\chi$ is given by\footnote{Here we correct a minor misprint in \cite{AblowitzSegur79}, namely the factor $2-\sigma^2$ in the numerator of the first fraction in their formula (2.24d) should be $3-\sigma^2$.}
	\[
	\chi=\frac{1}{4\sqrt{\sigma(1+\tilde T)}} \left(\frac{(1-\sigma^2)(9-\sigma^2)+\tilde T(3-\sigma^2)(7-\sigma^2)}{\sigma^2-\tilde T(3-\sigma^2)}
+8\sigma^2-2(1-\sigma^2)^2(1+\tilde T)-\frac{3\sigma^2 \tilde T}{1+\tilde T}
\right),
	\]
	in which $\tilde T=k^2 T/(g\rho)$ and $\sigma=\tanh(kh)$. Comparing with \eqref{eq-m212-formula}, we find that
	\[
	\widetilde m_{21}^{(2)}=\frac{32 k_* h  \sqrt{\sigma(1+\tilde T)}}{(1+\tilde T)(1-\sigma^2)^2} \chi.
	\]
	Hence $\widetilde m_{21}^{(2)}$ and $\chi$ have the same sign and give rise to the same instability criterion.
	The fact that the sign changes precisely along the second harmonic resonance curve $\Gamma_2$, where $k_{*,2}=2k_{*, 1}$, agrees with Figure 1 of \cite{AblowitzSegur79}.

	\newpage
	\bibliographystyle{siam}
	\bibliography{transverse_insta}
	
	\Addresses
	
\end{document}